\documentclass[12pt,a4paper,psamsfonts,reqno]{amsart}
\usepackage{amssymb}
\usepackage{amsfonts}
\usepackage{latexsym}
\usepackage{amscd}
\usepackage[mathscr]{euscript}
\usepackage{xy} \xyoption{all}

\newcommand{\N}{{\mathbb{N}}}

\newcommand{\Z}{{\mathbb{Z}}}

\newcommand{\ol}{\overline}
\newcommand{\uloopr}[1]{\ar@'{@+{[0,0]+(-4,5)}@+{[0,0]+(0,10)}@+{[0,0] +(4,5)}}^{#1}}
\newcommand{\uloopd}[1]{\ar@'{@+{[0,0]+(5,4)}@+{[0,0]+(10,0)}@+{[0,0]+ (5,-4)}}^{#1}}
\newcommand{\dloopr}[1]{\ar@'{@+{[0,0]+(-4,-5)}@+{[0,0]+(0,-10)}@+{[0, 0]+(4,-5)}}_{#1}}
\newcommand{\dloopd}[1]{\ar@'{@+{[0,0]+(-5,4)}@+{[0,0]+(-10,0)}@+{[0,0 ]+(-5,-4)}}_{#1}}

\newcommand{\luloop}[1]{\ar@'{@+{[0,0]+(-8,2)}@+{[0,0]+(-10,10)}@+{[0, 0]+(2,2)}}^{#1}}

\newtheorem{lem}{Lemma}[section]
\newtheorem{corol}[lem]{Corollary}
\newtheorem{theor}[lem]{Theorem}
\newtheorem{prop}[lem]{Proposition}
\newtheorem{rema}[lem]{Remark}
\newtheorem{defi}[lem]{Definition}

\newtheorem{point}[lem]{}

\newtheorem*{nota}{Notation}

\newcommand{\dnw}{\mathbin{\downarrow}}

\newcommand{\Ma}{{\rm Max}}

\DeclareMathOperator{\rL}{L}

\numberwithin{equation}{section}

   \usepackage{color}

   \usepackage{color}

\begin{document}

\title[Refinement Monoids]{Primely generated refinement monoids}
\author{P. Ara}
\address{Departament de Matem\`atiques, Universitat Aut\`onoma de Barcelona,
08193 Bellaterra (Barcelona), Spain.} \email{para@mat.uab.cat}
\author{E. Pardo}
\address{Departamento de Matem\'aticas, Facultad de Ciencias, Universidad de C\'adiz,
Campus de Puerto Real, 11510 Puerto Real (C\'adiz),
Spain.}\email{enrique.pardo@uca.es}
\urladdr{https://sites.google.com/a/gm.uca.es/enrique-pardo-s-home-page/} 

\thanks{The first-named 
author was partially supported by DGI MINECO
MTM2011-28992-C02-01, by FEDER UNAB10-4E-378
``Una manera de hacer Europa", and by the
Comissionat per Universitats i
Recerca de la Generalitat de Catalunya. The second-named author was partially supported by the DGI and European Regional Development Fund, jointly, through Project MTM2011-28992-C02-02, and by PAI III grants FQM-298 and P11-FQM-7156 of the Junta de Andaluc\'{\i}a.}
%
\dedicatory{Dedicated to the memory of Mar\'{\i}a Dolores Gordillo Romero}

\maketitle

\begin{abstract}
We extend both Dobbertin's characterization of primely generated regular refinement monoids and Pierce's characterization of primitive monoids to  
general primely generated refinement monoids. 
\end{abstract}

\section*{Introduction}\label{Introduction}

The class of abelian monoids satisfying the Riesz refinement property --refinement monoids for short-- has been largely studied over the last decades 
in connection with various problems, as non-stable K-Theory of rings and $C^*$-algebras (see e.g. \cite{Aposet, AGOP, AMFP, GPW, PaWe}), classification of Boolean 
algebras (see e.g. \cite{Ket}, \cite{Pierce}), or its own structure theory (see e.g. \cite{Brook, Dobb84, Whr3}). 
Recall that an element $p$ in a monoid $M$ is a {\it prime element} if $p$ is not invertible in $M$, and, whenever 
$p\leq a+b$ for $a,b\in M$, then either $p\leq a$ or $p\leq b$ (where $x\le y$ means that $y=x+z$ for some $z\in M$). 
The monoid $M$ is {\it primely generated} if every non-invertible element of $M$ 
can be written as a sum of prime elements. Primely generated refinement monoids enjoy important cancellation properties, such as separative cancellation and 
unperforation, as shown by Brookfield in \cite[Theorem 4.5, Corollary 5.11(5)]{Brook}. It was also shown by Brookfield that any finitely 
generated refinement monoid is automatically primely generated \cite[Corollary 6.8]{Brook}. 

\medskip

Recently, the class of refinement monoids has been separated into subclasses of {\it tame} and {\it wild} refinement monoids, where the tame ones are the direct limits
of finitely generated refinement monoids, and the rest are wild. This has been motivated by problems in non-stable K-theory, where many of the monoids that appear 
in connection with von Neumann regular rings and C*-algebras of real rank zero are indeed tame refinement monoids. 
It has been asked in \cite[Open Problem 5.3]{AraGood} whether all primely generated refinement monoids are tame. 

\medskip 

Two classes of primely generated refinement monoids have been completely classified. The first one is the class of {\it primitive monoids}, i.e. antisymmetric
primely generated refinement monoids, see \cite{Pierce}. These monoids are described by means of a set $I$ endowed with an antisymmetric transitive relation $\vartriangleleft$. 
Given such a pair $(I,\vartriangleleft)$, the primitive monoid associated to it is the monoid generated by $I$ with the relations $i+j= j$ if and only if $i\vartriangleleft j$.
We observe that this is the same as giving a partial order $\le $ on $I$ and a decomposition $I=I_{{\rm free}}\sqcup I_{{\rm reg}}$ into {\it free} and {\it regular}
elements, the free elements  corresponding to the elements $i$ such that $i\ntriangleleft i$, and the regular ones corresponding to the elements $i$ such that
$i \vartriangleleft i$. Using this structure, tameness of primitive monoids has been verified in \cite[Theorem 2.10]{AraGood}. The second class where a satisfactory 
description has been obtained is that of primely generated {\it regular} conical refinement monoids.
These monoids were characterized by Dobbertin in \cite{Dobb84} in terms of partial orders of abelian groups.\footnote{It should be pointed out that the definition of refinement monoid given in \cite{Dobb84} includes the condition 
of being {\it conical} (condition (1) in \cite[p. 166]{Dobb84}). Following modern convention, refinement monoids are defined here just in terms of the
accomplishment of the Riesz refinement property (see Section \ref{section:Isystems} for the precise definitions of these conditions).}
It has been shown by the second-named author and Wehrung \cite[Theorem 4.4]{PWunp} that all regular conical refinement monoids are tame.

\medskip 

In the present paper, we obtain a common generalization of both
results, obtaining a representation of primely generated conical refinement monoids in terms of certain partial orders of semigroups. The basic data are a poset $I$, together with a  
partition $I=I_{{\rm free}}\sqcup I_{{\rm reg}}$, a family of abelian groups $G_i$ for $i\in I_{{\rm reg}}$, and a family of semigroups of the form
$\N\times G_i$, where $G_i$ is an abelian group, for $i\in I_{{\rm free}}$ (see Definition \ref{def:I-system} below for the precise definition). To each one of these {\it $I$-systems}
$\mathcal J$ we associate a conical monoid $M(\mathcal J )$.  

\medskip

With this notation and terminology at hand, we can state the main results of the paper as follows:

\begin{theor}
 \label{thm:first-main}
 $\mbox{ }$
 
\begin{enumerate}
 \item  Given any primely generated conical refinement monoid $M$, there is a poset $I$ and an $I$-system $\mathcal J$ such that $M\cong M(\mathcal J)$.
 \item For any $I$-system $\mathcal J$, the monoid $M(\mathcal J )$ is a primely generated conical refinement monoid. Moreover, $M(\mathcal J)$ is a tame monoid.  
\end{enumerate}
\end{theor}

Note that this result gives a complete description of primely generated conical refinement monoids. It also gives an affirmative answer to \cite[Open Problem 5.3]{AraGood}.
In the particular case where all the above groups $G_i$ are trivial, we recover Pierce's characterization of primitive monoids. 
In the case that $I_{{\rm free}}=\emptyset$, our result reduces to Dobbertin's characterization of primely generated regular conical refinement monoids. 
Theorem \ref{thm:first-main}(1) 
will be proven in Section \ref{section:theIsystemof-a-monoid} (see Theorem \ref{thm:M=M(J)}). 

\medskip

It is readily checked that $M(\mathcal J)$ is a primely generated conical monoid for every $I$-system $\mathcal J$ (Remark \ref{remark:primes}). 
It will be shown in Sections \ref{sect:ref-for-fg} and \ref{sect:ref-arbitrary} that $M(\mathcal J)$ is a tame refinement monoid, thus completing the proof of  
Theorem \ref{thm:first-main}. We point out that one of the main difficulties in showing the refinement property for $M(\mathcal J)$,  when $I_{{\rm free}} \ne \emptyset$,
lies in the fact that the archimedian components of this monoid do not satisfy refinement in general. 

\medskip

Section \ref{sect:ref-for-fg} is devoted to the proof of the refinement property of $M(\mathcal J)$ for finitely generated $I$-systems.
We show this result first in the case where all the upper subsets $I\uparrow i:=\{ j\in I: i\le j\} $, for $i\in I$,  are chains 
(Proposition \ref{prop:chins-up}), and then we 
adapt a technique introduced in \cite{Aposet} to solve the general case of a finitely generated $I$-system (Theorem \ref{theor:total-refinement}).
We believe that, in analogy with \cite{Aposet}, the methods developed in this section will be useful in the study of the realization problem for 
finitely generated conical refinement monoids (see \cite{Areal} for a survey on this problem).

\medskip

For a general $I$-system $\mathcal J$, we prove in Section \ref{sect:ref-arbitrary} that $M(\mathcal J)$ can be written as a direct limit of monoids of the form 
$M(\mathcal J' )$, where $\mathcal J'$ are finitely generated systems. By Proposition \ref{prop:finitely-generated} and Theorem \ref{theor:total-refinement}, 
all these monoids $M(\mathcal J ')$ are 
finitely generated conical refinement monoids, so we obtain at once that $M(\mathcal J)$ is a refinement monoid, and that it is tame.    

\medskip

We work with conical monoids because these are the monoids that appear in
non-stable K-theory. However, similar results can be obtained for non-conical refinement monoids $M$ 
by considering the conical refinement monoid $M\sqcup \{ 0 \}$ obtained 
by adjoining a new zero element to $M$.

\section{$I$-systems and their monoids}
\label{section:Isystems}

In this section we will define $I$-systems --a structure extending the notion of partial order of groups-- by replacing the groups by some special semigroups. Using this structure,
we will construct an associated monoid in a way that extends Dobbertin's construction \cite{Dobb84} and Pierce's construction \cite{Pierce}. First, we will recall some basic definitions. 

Given a poset $(I, \leq)$, we say that a subset $A$ of $I$ is a {\it lower set} if $x\leq y$ in $I$ and $y\in A$ implies $x\in A$. For any $i\in I$, 
we will denote by $I\downarrow i=\{x\in I : x\leq i\}$ the lower subset generated by $i$. 
We will write $x<y$ if $x\le y$ and $x\ne y$. 

All semigroups considered in this paper are abelian. We will denote by $\N$ the semigroup of positive integers, and by $\Z^+$ the monoid of non-negative integers.  

Given an abelian monoid $M$, we set $M^*:=M\setminus\{0\}$. We say that $M$ is {\it conical} if $M^*$ is a semigroup, that is, if, for all $x$, $y$ in $M$, $x+y=0$ only when $x=y=0$. We say that $M$ is {\it separative} 
provided $2x=2y=x+y$ always implies $x=y$; there are a number of equivalent formulations of this property, see e.g. \cite[Lemma 2.1]{AGOP}.
We say that  $M$ is
a {\it refinement monoid} if, for all $a$, $b$, $c$, $d$ in
$M$ such that $a+b=c+d$, there exist $w$, $x$, $y$, $z$ in $M$ such
that $a=w+x$, $b=y+z$, 
$c=w+y$ and $d=x+z$.  It will often be convenient to present this
situation in the form of a diagram, as follows:
$$\mbox{\begin{tabular}{|l|l|l|}
\cline{2-3}
\multicolumn{1}{l|}{} & ${c}$ & ${d}$ \\ \hline
${a}$ & ${w}$ & ${x}$ \\ \hline
${b}$ & ${y}$ & ${z}$ \\ \hline
\end{tabular}}.$$  
If $x, y\in M$, we write $x\leq y$ 
if there exists $z\in M$  such that $x+z = y$.
Note that $\le$ is a translation-invariant pre-order on $M$, called the {\it algebraic pre-order} of $M$. 
An element $x\in M$ is {\it regular} if $2x\leq x$. An element $x\in M$ is an {\it idempotent} if $2x= x$. An element 
$x\in M$ is {\it free} if $nx\leq mx$ implies $n\leq m$. Any element of a separative monoid is either free or regular.
In particular, this is the case for any primely generated refinement monoid, by \cite[Theorem 4.5]{Brook}.

A subset $S$ of a monoid $M$ is called an {\it order-ideal} if $S$ is a subset of $M$ containing $0$,
closed under taking sums and summands within $M$;  that is,
$S$ is a submonoid such that, for all $x \in M$ and $e
\in S$, if $x\leq e$ then $x \in S$.  If $(S_k)_{k\in \Lambda}$ is a family of (commutative) semigroups, $\bigoplus _{k\in \Lambda} S_k$ (resp. $\prod _{k\in \Lambda} S_k$) 
stands for the coproduct (resp. the product) of the semigroups $S_k$, $k\in \Lambda$, in the category of commutative
semigroups. If the semigroups $S_k$ are subsemigroups of a semigroup $S$, we will denote by $\sum\limits_{k\in \Lambda} S_k$ the subsemigroup of $S$ generated by $\bigcup_{k\in \Lambda}S_k$.
Note that $\sum\limits_{k\in \Lambda} S_k$ is the image of the canonical map $\bigoplus_{k\in \Lambda} S_k\to S$. We will use the notation $\langle X\rangle $ to denote the semigroup
generated by a subset $X$ of a semigroup $S$.

Given a semigroup $M$, we will denote by $G(M)$ the Grothendieck group of $M$. There exists a semigroup homomorphism $\psi_M\colon M\to G(M)$
such that for any semigroup homomorphism $\eta \colon M\to H$ to a group $H$ there is a unique group homomorphism $\widetilde{\eta}\colon G(M)\to H$ such that 
$\widetilde{\eta}\circ \psi_M= \eta$. $G(M)$ is abelian and it is generated as a group by $\psi (M)$.  If $M$ is already a group then $G(M)= M$. If $M$ is a semigroup of the form 
$\N\times G$, where $G$ is an abelian group, then $G(M)= \Z\times G$. In this case, we will view $G$ as a subgroup of $\Z\times G$ by means of the identification $g\leftrightarrow (0,g)$. These are the only cases where we will need to consider Grothendieck groups in this article.

\vspace{.3truecm}

The following definition is crucial for this work:

\begin{defi}
\label{def:I-system} {\rm Let $I= (I,\le )$ be a poset.  An {\it $I$-system} $\mathcal{J}=\left(I, \leq , (G_i)_{i\in I}, \varphi_{ji} \, (i<j)\right)$ is given by the following data:
\begin{enumerate}
\item[(a)] A partition
$I=I_{free}\sqcup I_{reg}$ (we admit one of the two sets
$I_{free}$ or $I_{reg}$ to be empty).
\item[(b)] A family $\{G_i\}_{i\in I}$ of abelian groups. We adopt the following notation: 
\begin{itemize}
\item[(1)] For $i\in I_{reg}$, set $M_i = G_i$, and $\widehat{G}_i=G_i=M_i$.
\item[(2)] For $i\in I_{free}$, set $M_i=\N \times G_i$, and $\widehat{G}_i= \Z\times G_i$
\end{itemize}
Observe that, in any case, $\widehat{G}_i$ is the Grothendieck group of $M_i$.
\item[(c)]
A family of  semigroup homomorphisms $\varphi _{ji}\colon M_i\to
G_j$ for all $i<j$, to whom we associate, for all $i<j$, the unique extension $\widehat{\varphi}_{ji}\colon \widehat{G}_i \to G_j$ of $\varphi _{ji}$ to a group homomorphism
from the Grothendieck group of $M_i$  to $G_j$ (we look at these
maps as maps from $\widehat{G}_i$ to $\widehat{G}_j$). We require that the family $ \{ \varphi_{ji} \}$ satisfies the following conditions:
\begin{itemize}
\item[(1)]  The assignment
$$
\left\{
\begin{array}{ccc}
i & \mapsto  &  \widehat{G}_i   \\
 (i<j) & \mapsto  &   \widehat{\varphi}_{ji}
\end{array}
\right\}
$$ 
defines a functor from the
category $I$ to the category of abelian groups (where we set $\widehat{\varphi}_{ii}= {\rm  id}_{\widehat{G}_i}$
for all $i\in I$).
\item[(2)]  For each $i\in I_{free}$ we have that the map
$$\bigoplus _{k<i} \varphi _{ik}\colon  \bigoplus _{k<i}  M_k \to G_i$$
is surjective.
\end{itemize}
\end{enumerate}
We say that an $I$-system $\mathcal  J =  \left(I, \leq , (G_i)_{i\in I}, \varphi_{ji}\, (i<j)\right)$ is {\it finitely generated} 
in case $I$ is a finite poset and all the groups $G_i$ are finitely generated.} 
\end{defi}

\begin{rema}
 \label{rem:minimals-in-I} 
{\rm  Note that, if $i\in I_{{\rm free}}$ and $i$ is a minimal element of $I$, then $G_i=\{ e_i \}$, and so $M_i=\N$, by condition (c2) in Definition \ref{def:I-system}.
Indeed the map appearing there in this special case should be interpreted as the map $\{0\}\to  \{ e_i\} $. For $i$ being not minimal, the inclusion of $\{ 0\}$ in the domain of the map 
makes no difference.}
 \end{rema}

Given a poset $I$, and an $I$-system $\mathcal{J}$, we construct a semilattice of groups based on the partial order 
of groups $(I,  \le, \widehat{G}_i)$, by following the model introduced in  \cite{Dobb84}. Let $A(I)$ be
the semilattice (under set-theoretic union) of all the finitely generated lower subsets
of $I$. These are precisely the lower subsets $a$ of $I$ such that the set $\Ma (a)$ of maximal elements of $a$ is finite and each element of $a$ is under some of the maximal ones.
In case $I$ is finite, and since the intersection of lower subsets of $I$ is again a lower subset, $A(I)$ is a lattice. 
For any $a\in A(I)$, we define $\widehat{H}_a= \bigoplus _{i\in a}
\widehat{G}_i$, and we define $f^b_a$ $(a\subseteq b)$ to be the canonical
embedding of $\widehat{H}_a$ into $\widehat{H}_b$. Given $a\in A(I), i\in a$ and $ u\in \widehat{G}_i$, we define $\chi (a,i,u)\in \widehat H _a$ by 
$$\chi (a,i,u)_j=
\left\{
\begin{array}{cc}
  u   & \mbox{if }j=i,  \\
   0_j  &   \mbox{if }j\ne i. 
\end{array}
\right.
$$
Let $U_a$ be the subgroup of $\widehat{H}_a$
generated by the set 
$$\{\chi (a,i,u) -\chi (a,j,\widehat{\varphi}_{ji}(u)) :  i<j\in \Ma (a), u\in \widehat{G}_i\}.$$

Now, for any $a\in A(I)$, set $\widetilde{G}_a= \widehat{H}_a/U_a$, and let $\Phi_a:\widehat{H}_a\rightarrow \widetilde{G}_a$ be the natural onto map. Then, for any $a\subseteq b\in A(I)$ we have that $f_a^b(U_a)\subseteq U_b$, so that there exists a unique homomorphism $\widetilde{f}_a^b: \widetilde{G}_a\rightarrow \widetilde{G}_b$ which makes the diagram 
$$
\xymatrix{\widehat{H}_a\ar[r]^{f_a^b} \ar[d]_{\Phi_a} & \widehat{H}_b\ar[d]^{\Phi_b}\\
\widetilde{G}_a \ar[r]_{\widetilde{f}_a^b}  & \widetilde{G}_b}
$$
commutative. Hence,  $(A(I), (\widetilde{G}_a)_{a\in A(I)}, \widetilde{f}_a^b (a\subset b))$ is a semilattice of groups. Thus, the set 
$$\widetilde{M}(\mathcal J ):=\bigsqcup\limits_{a\in A(I)}\widetilde{G}_a,$$ endowed with the operation 
$x+y:= \widetilde{f}_a^{a\cup b}(x)+\widetilde{f}_b^{a\cup b}(y)$ for any $a,b\in A(I)$ and any $x\in \widetilde{G}_a, y\in \widetilde{G}_b$, 
is a primely generated regular refinement monoid by \cite[Proposition 1]{Dobb84}. Note that
$\widehat{H}_{\emptyset}=\widetilde{G}_{\emptyset}=\{ 0 \}$. We refer the reader to \cite{Dobb84} for further details on this
construction.

In order to attain our goal, we define a convenient substructure of $\widehat{H}_a$. Let $H_a$ be
the subsemigroup of $\widehat{H}_a$ defined by
$$H_a= \left\{ (z_i)_{i\in a} \in \widehat{H}_a : z_i\in \left\{ 
\begin{array}{ccc}
 \N\times G_i  & \text{ for }  &  i\in \Ma
(a)_{{\rm free}} \\
 \{ (0,0_i) \} \cup (\N\times G_i) & \text{ for }  &  i\in a_{{\rm free}}\setminus \Ma (a)_{{\rm free}} 
\end{array}
\right.
 \right\}.
$$

In what follows, whenever $i<j \in I$ with $j$ a free element, $x= (n,g)\in \N \times G_j$ and $y\in M_i$,
we will see $x+\varphi _{ji}(y)$ as the element $(n, g+\varphi _{ji}(y))\in \N\times G_j$. This is coherent with our identification of $G_j$ as
the subgroup $\{0\} \times G_j$ of $\widehat{G}_j= \Z\times G_j$. 

\begin{lem}
\label{lem:well-defsum} If $x\in H_a$ and $y\in H_b$, then
$$f_a^{a\cup b} (x) +f_b^{a\cup b} (y) \in H_{a\cup b}.$$
\end{lem}
\begin{proof}
This follows from the fact that $\Ma (a\cup b) \subseteq \Ma (a )
\cup \Ma (b) $.
\end{proof}

Lemma \ref{lem:well-defsum} implies that we can define a semilattice of semigroups
$$(A(I), (H_a)_{a\in A(I)}, f_a^b (a\subset b)).$$ 
Now, we will construct a monoid associated to it. For this, consider the congruence $\sim $ defined on $H_a$, for $a\in A(I)$, given by
$$x\sim y \iff x-y\in U_a .$$

\begin{lem}
\label{lem:equivrels} Let $a\in A(I)$. The congruence $\sim $ on
$H_a$ agrees with the congruence $\equiv$, generated by the pairs $(x+\chi
(a,i,\alpha), x+\chi (a,j, \varphi _{ji}(\alpha) ))$, for $x\in
H_a$,  $i<j\in \Ma (a)$ and $\alpha \in M_i$.
\end{lem}
\begin{proof} It is clear that if
$x\equiv y$ then $x\sim y$.

Assume that $x\sim y$. Then there exist a finite subset $A$ consisting of pairs $(i,j)\in I^2$ such that $i<j$ for all $(i,j)\in A$ and $j\in \Ma
(a)$, and elements $\alpha ^{(j)}_i, \beta^{(j)}_i\in M_i$ such that
\begin{equation}
\label{eq:ident1} x+\sum\limits _{(i,j)\in A}  (\chi(a,j, \varphi
_{ji}(\alpha^{(j)}_i))+\chi(a,i, \beta^{(j)}_i))= y+\sum\limits _{(i,j)\in A}  (\chi(a,j, \varphi _{ji}(\beta^{(j)}_i))+\chi(a,i,
\alpha^{(j)}_i)).
\end{equation}
Observe that
$$x+\sum\limits _{(i,j)\in A}  (\chi(a,j, \varphi _{ji}(\beta^{(j)}_i))+\chi(a,i,
\alpha^{(j)}_i))\equiv y+\sum\limits _{(i,j)\in A}  (\chi(a,j, \varphi
_{ji}(\beta^{(j)}_i))+\chi(a,i, \alpha^{(j)}_i))\, ,$$ so that it is
enough to show that for all $z=\chi (a,i,\alpha)$ and all $z= \chi
(a,j, \varphi _{ji}(\alpha))$, $i<j\in \Ma (a)$, $\alpha\in M_i$,
and for all $x,y\in H_a$, we have $x+z\equiv y+z\implies x\equiv y$.

Assume that for $x,y\in H_a$, $i<j\in \Ma (a)$ and $\alpha \in M_i$ we have $x+\chi (a,i,\alpha)\equiv y+\chi (a,i,\alpha)$. We then have
$$x+\chi (a,j,\varphi_{ji}(\alpha))\equiv x+\chi(a,i,\alpha)\equiv
y+\chi(a,i,\alpha) \equiv y+\chi (a,j,\varphi _{ji}(\alpha)) ,$$ so
that it suffices to show that $x+\chi(a,j,u)\equiv y+\chi (a,j,u)$
implies $x\equiv y $ for $u\in G_j$, $j\in \Ma (a)$. But this is
easy: let $x^{(0)},x^{(1)},\dots ,x^{(n)}$ be elements in $H_a$ such
that $x^{(0)}= x+\chi(a,j,u)$  and $x^{(n)}= y+\chi (a,j, u)$, and
such that, for $m=0,\dots ,n-1$,  either $x^{(m)}=(x^{(m)})'+\chi
(a,k,\alpha_m)$ and $x^{(m+1)}= (x^{(m)})'+\chi(a,l,
\varphi_{lk}(\alpha_m))$, or $x^{(m)}=(x^{(m)})'+\chi
(a,l,\varphi_{lk}(\alpha_m))$ and $x^{(m+1)}= (x^{(m)})'+\chi(a,l,
\alpha_m)$,  for some $(x^{(m)})'\in H_a$, $\alpha _m\in M_k$, and
some pair $k<l\in \Ma (a)$. Then, setting 
$$z^{(m)}= x^{(m)}-\chi (a,j,u) = x^{(m)} +\chi (a,j, -u)\in H_a,\qquad (m=0,1, \dots ,n)$$
we get that $z^{(0)}= x$, $z^{(n)}= y$ and $z^{(m)}$ are elements in $H_a$ satisfying relations analogous 
to the ones satisfied by $x^{(m)}$
(with $(z^{(m)})':= (x^{(m)})'-\chi (a,j,u)$ in place of $(x^{(m)})'$). 
This shows that $x\equiv y$.
\end{proof}

\begin{corol}\label{cor:NouSubmonoid}
For every $a\in A(I)$, $M_a:=H_a/\sim = H_a/\equiv $ is a submonoid of $\widetilde{G}_a$.
\end{corol}
\begin{proof}
By Lemma \ref{lem:equivrels}, we have $H_a/\sim = H_a/\equiv $. 
Clearly, $M_a= H_a/\sim$ is a submonoid of $\widehat{H}_a/\sim$.
\end{proof}

\begin{defi}
\label{def:MJ} {\rm Given an $I$-system $\mathcal J=(I,\le, G_i,\varphi _{ji} (i<j))$, we denote by
$M(\mathcal J)$ the set $\bigsqcup _{a\in A(I)} M_a$. By Lemma \ref{lem:well-defsum} and 
Corollary \ref{cor:NouSubmonoid}, $M(\mathcal J)$  is a submonoid of $\widetilde{M}(\mathcal J)$.}
\end{defi}

\begin{rema}\label{rem:NouMonoid}
{\rm $\mbox{ }$
If $I=I_{reg}$, then $M(\mathcal{J})$ is the monoid constructed by Dobbertin \cite{Dobb84}. If all groups $G_i$ are trivial, then we recover
Pierce's primitive monoids \cite{Pierce}. }
\end{rema}

Observe that Lemma \ref{lem:equivrels} gives:

\begin{corol}
\label{cor:presentationofM} $M(\mathcal{J})$ is the
monoid generated by $M_i$, $i\in I$, with respect to the defining
relations
$$x+y= x+\varphi _{ji}(y), \quad i<j, \, x\in M_j,\,  y\in M_i.$$
\end{corol}

\begin{proof}
 This follows from the fact that $M_a= H_a/\equiv $ for all $a\in A(I)$ (Lemma \ref{lem:equivrels}).
 \end{proof}

\begin{nota}{\rm
\label{nota:chisubi} Assume $\mathcal J$ is an $I$-system. For $i\in I$ and $x\in M_i$ we will denote by
$\chi _i(x)$ the element $[\chi (I\dnw i, i, x)]\in M(\mathcal J)$. Note that, by Corollary \ref{cor:presentationofM},
$M(\mathcal J )$ is the monoid generated by $\chi _i (x)$, $i\in I$, $x\in M_i$, with the defining relations
$$\chi _j (x)+\chi _i (y)= \chi _j (x+\varphi _{ji}(y)), \quad i<j, \, x\in M_j,\,  y\in M_i.$$}
\end{nota}

We will denote by $\mathcal L (M)$ the lattice of order-ideals of a monoid $M$ and by $\mathcal L (I)$ the lattice of lower subsets
of a poset $I$.

\begin{prop}
 \label{prop:characideals}
 Let $\mathcal J$ be an $I$-system. Then there is a lattice isomorphism $$\mathcal L (I) \cong \mathcal L (M(\mathcal J)).$$
 More precisely, given a lower subset $J$ of $I$, the restricted $J$-system is $$\mathcal J _J :=(J, \le , (G_i)_{i\in J}, \varphi _{ji}, (i<j\in J)), $$
 and the 
 map $J\mapsto M(\mathcal J _J)$ defines a lattice isomorphism from $\mathcal L (I)$ onto $\mathcal L (M(\mathcal J ))$. 
 \end{prop}

 \begin{proof}
  Since $J$ is a lower subset, we see that $\mathcal J _J$ is a $J$-system (of course, we set $J_{{\rm free}} = J\cap I_{{\rm free}}$ and
  $J_{{\rm reg}} = J\cap I_{{\rm reg}}$).  Given $a\in A(J)$, we also have that $a\in A(I)$, and moreover $M_a$ only depends on the system restricted to $a$, 
  therefore we get an embedding 
  $$M(\mathcal J _J) = \bigsqcup_{a\in A(J)} M_a \hookrightarrow  M(\mathcal J) =  \bigsqcup_{a\in A(I)} M_a.$$
  Clearly $M(\mathcal J _J)$ is an order-ideal of $M(\mathcal J )$. 
  
  The map $J\mapsto M(\mathcal J _J)$ is clearly injective. To show surjectivity, let $N$ be an order-ideal of $M(\mathcal J )$, and let $J$ be the subset of elements $i$ of $I$
  such that $\chi _i (x)\in N$ for some $x\in M_i$. Observe that, for $z\in M_i$, the archimedian component of $\chi_i(z)$ in $M(\mathcal J)$ is precisely $\chi_i (M_i)$.
  If $i\in J$, then there is $x\in M_i$ such that $\chi_i(x)\in N$. Hence, since $N$ is an order-ideal, it must contain the archimedian component of $\chi_i(x)$, that is,
  $\chi_i(M_i)\subseteq N$. To show that $J$ is a lower subset of $I$, take $i,j\in I$ with $j<i$ and $i\in J$. Then, for $y\in M_j$ and $x\in M_i$, 
  we have $\chi_i(x)+\chi_j(y) = \chi_i (x+\varphi_{ij}(y))\in \chi_i(M_i) \subseteq N$,
  and since $N$ is an order-ideal we get that $\chi_j(y)\in N$, showing that $j\in J$.  
    We claim that $M(\mathcal J_J) = N$. If $x\in N$ then there is $a\in A(I)$ such that  $x= \sum\limits_{i\in \Ma (a)} \chi_i(x_i)$ for some 
  $x_i \in M_i$. Since $N$ is an order-ideal, we get $\chi _i(x_i)\in N$, and so $i\in J$. This shows that $N\subseteq M(\mathcal J_J)$.
  Conversely, $M(\mathcal J_J)$ is generated as a monoid by the elements $\chi _i (x)$ for $i\in J$ and $x\in M_i$, so it suffices to show that all these elements belong to $N$. 
  If $i\in J$ then there is an element $z\in M_i$ such that $\chi _i (z)\in N$. Since the archimedian component of $\chi_i (z)$ in $M(\mathcal J)$ is  
  $\chi _i (M_i)$, we get that $\chi _i (M_i)\subseteq N$, as desired.
  
We have shown that the map $J\mapsto M(\mathcal J _J)$ is a bijection. It is easily checked that this map is a lattice isomorphism.
  \end{proof}

 Observe that the map defined in Proposition \ref{prop:characideals} restricts to a semilattice isomorphism from $A(I)$ to the semilattice of finitely generated 
 order-ideals of $M(\mathcal J)$.
 
\vspace{.2truecm}

\section{The $I$-system of a primely generated refinement monoid}
\label{section:theIsystemof-a-monoid}

In this section we will show that for any primely generated conical refinement monoid $M$ there exist a poset 
$I$ and an $I$-system $\mathcal{J}_M$ such that $M$ and $M(\mathcal{J}_M)$ are isomorphic.

The set of primes of an abelian monoid $M$ is denoted by $\mathbb{P}(M)$. 
Two primes $p,q\in M$ are {\it incomparable} if $p\nleq q$ and $q\nleq p$. 
Let $\ol{M}$ be the antisymmetrization of $M$, i.e. the quotient monoid 
of $M$ by the congruence given by $x\equiv y$ if and only if $x\leq y$ and $y\leq x$ (see \cite[Notation 5.1]{Brook}). We will denote the class of an 
element $x$ of $M$ in $\ol{M}$ by $\ol{x}$. 

The following facts are easily proven.

\begin{lem}\label{lem:primeslift}
If $p\in M$, then:
\begin{enumerate}
\item $\ol{p}\in \ol{M}$ is regular if and only if so is $p$.
\item $\ol{p}\in \ol{M}$ is free if and only if so is $p$.
\item $\ol{p}\in \ol{M}$ is prime if and only if so is $p$.
\end{enumerate}
\end{lem}

For $x,y\in M$, we will write $x<^* y$ when $\ol{x}< \ol{y}$ in $\ol{M}$. We write $x\leq^* y$ if either $x<^*y$ or $x = y$. 

\begin{lem}\label{lem:Tech1}
Let $p, q\in M$ be primes, and suppose that $q< ^* p$. Then, $\ol{p}+\ol{q}=\ol{p}$.
\end{lem}
\begin{proof}
We have $\ol{p}=\ol{q}+\ol{a}$ for a nonzero $\ol{a}\in \ol{M}$, and thus either $\ol{p}\leq \ol{q}$ or $\ol{p}\leq \ol{a}$. The first case implies that 
$\ol{p}=\ol{q}$, contradicting the assumption. Hence, $\ol{p}+\ol{q}\leq \ol{a} +\ol{q}=\ol{p}$, and thus $\ol{p}+\ol{q}=\ol{p}$, as desired.
\end{proof}

For $a\in M$, we denote by $M_a$ the archimedian component of $a$, so that $x\in M_a$ if and only if $a\le nx$ and $x\le ma$ for some positive integers $n,m$. 

\begin{lem}\label{lem:Tech2}
Let $M$ be a primely generated refinement monoid, and let $p\in \mathbb{P}(M)_{{\rm free}}$. If $x\in M_p$, then there exists a unique $n\in \N$ such that $\ol{x}=n\ol{p}$.
\end{lem}
\begin{proof} By \cite[Theorem 5.8]{Brook}, there are unique pairwise incomparable primes $\ol{q}_1,\dots , \ol{q}_s$ and uniquely determined positive integers $n_1,\dots , n_s$, with $n_i=1$ if $\ol{q}_i$ is regular,
such that 
$$\ol{x} = n_1\ol{q}_1+\cdots + n_s\ol{q}_s .$$
Since $\ol{x}\le n\ol{p}$ for some $n\in \N$, we get that $\ol{q}_i\le \ol{p}$ for all $i=1,\dots , s$, and since $\ol{p}\le m \ol{x}$ for some $m \in \N$ we must have $\ol{p}\le \ol{q_i}$ for some $i$. Therefore 
$\ol{p}= \ol{q}_i$ and, since the $\ol{q}_j$ are incomparable we must have $s=i=1$. This gives the result. 
\end{proof}

We are now ready to define the $I$-system associated to a primely generated conical refinement monoid $M$:\vspace{.3truecm}

\noindent (1) By \cite[Theorem 5.2]{Brook}, $\ol{M}$ is a  primitive monoid. We will choose, for each prime $\ol{p}$ of $\ol{M}$, a representative $p$ of $\ol{p}$ in $M$, and we will consider the set $\mathbb P$ formed 
by the set of all the elements $p$ obtained in this way; notice that, by Lemma \ref{lem:primeslift}, $\mathbb{P}\subseteq\mathbb{P}(M)$. We will refer to these 
elements as the primes of $M$, although any element $p'$ such that $\ol{p}=\ol{ p'}$ will be also prime of course. Note that, again by Lemma \ref{lem:primeslift}, $p$ is 
regular or free according to whether $\ol{p}$ is regular or free in $\ol{M}$. 

The chosen poset is $\mathbb{P}$ endowed with the partial order $\leq^*$. Note that $(\mathbb P, \le^*)$ is order-isomorphic with
$(\ol{\mathbb P},\le )= (\mathbb P (\ol{M}), \le )$.

\vspace{.2truecm}

\noindent (2) For each $p\in \mathbb P$, let $M_p$ be the archimedian component
of $p$. We separate two cases:
\begin{enumerate}
\item[(i)] If $p$ is regular then $M_p$ is an abelian group (see e.g. \cite[Lemma 2.7]{Brook}), denoted by
$G_p$. In this case, we choose as the
canonical representative of $\ol{p}$ the only idempotent element lifting $\ol{p}$, i.e. the unit of the group $M_p$.
\item[(ii)] If $p$ is free, we define
$$G'_p= \{ p+\alpha : \alpha \in M \text{ and } p+\alpha \le p \}.$$
Then $G'_p$ is a group with respect to the operation $\circ $ given
by:
$$(p+\alpha) \circ (p+\beta ) = p+(\alpha +\beta )$$
(see \cite[Definition 2.8 \& Lemma  2.9]{Brook}). 
\end{enumerate}

In order to complete the picture for the case (ii), we need to prove the following result.

\begin{lem}
\label{lem:archcomp} Let $p\in \mathbb P$ be a free prime. Then
$$M_p\cong \N\times (G'_p,\circ ).$$
\end{lem}
\begin{proof}
If $p+\alpha\in G'_p$, then for any $n\in \N$ we have $np+\alpha\in M_p$. Thus, we can define a map $\varphi \colon \N\times G'_p\to M_p$ by the rule
$$\varphi (n, p+\alpha )= np+\alpha.$$
 
We have $$\varphi
((n,p+\alpha )+(m,p+\beta ))=\varphi (n+m,p+\alpha +\beta)=
(n+m)p+\alpha +\beta=\varphi (n,p+\alpha )+\varphi (m,p+\beta),$$
showing that $\varphi $ is a homomorphism.

Now we show that $\varphi $ is injective. Suppose that $np+\alpha
=mp+\beta $ for some $n\in \N $ and $p+\alpha,p+\beta \in G'_p$. Passing
to $\ol{M}$ and using that $p$ is free, we get $n=m$. Now by
separative cancellation (\cite[Theorem 4.5]{Brook}), we get $p+\alpha
=p+\beta$.

To show that $\varphi$ is surjective, take any $x\in M_p$. By Lemma \ref{lem:Tech2},
$\ol{x}=n\ol{p}$ for some positive integer $n$, so that $x+\alpha
=np$ and $np+\beta = x$ for some $\alpha, \beta \in M$. We thus
obtain $np= np+\alpha +\beta$, which again by separative
cancellation gives $p=p+\alpha +\beta$, showing that $p+\beta \in
G'_p$. We obtain $x= \varphi (n, p+\beta)$.
\end{proof}

\begin{rema}
 \label{rem:psi-iso}
{\rm When $p$ is a free prime, the group $(G'_p,\circ)$ can be identified with the subgroup $G_p:=\{(p+\alpha)-p : p+\alpha\in G'_p\}$ of $G(M_p)$ through the group isomorphism
$$
\begin{array}{crcl}
\phi : & G'_p  &\rightarrow   & G_p  \\
 & p+\alpha & \mapsto  &   (p+\alpha)-p
\end{array}
$$
On the other hand, Lemma \ref{lem:archcomp} gives an isomorphism 
$$
\begin{array}{crcl}
\varphi : & \N\times (G'_p,\circ )   &\rightarrow   & M_p \\
 & (n, p+\alpha) & \mapsto  &   np+\alpha \, ,
\end{array}
$$
so that we get an isomorphism
$$
\begin{array}{crcl}
\psi_p : & \N\times G_p   &\rightarrow   & M_p \\
 & (n, (p+\alpha) -p) & \mapsto  &   np+\alpha \, .
\end{array}
$$
Since $G(\N\times G_p) = \Z\times G_p$, we obtain an isomorphism
$$
\begin{array}{crcl}
G(\psi_p) : & \Z\times G_p   &\rightarrow   & G(M_p) \\
 & (n-m, (p+\alpha) -p) & \mapsto  &   (np+\alpha) -mp  \, ,
\end{array}
$$
for $n,m\in \N$ and $(p+\alpha )- p \in G_p$. 
We will freely use these identifications in the sequel.
} 
\end{rema}

The groups $(G_p)_{p\in \mathbb{P}}$ constructed above are the groups for our $\mathbb{P}$-system.\vspace{.2truecm}

\noindent (3) We now define maps $\varphi _{pq}\colon M_q\to G_p$ for $q<^*p$,
$p,q\in \mathbb P$. First we need the following fact

\begin{lem}\label{lem:elkefaltava}
If $q<^* p\in \mathbb{P}$ and $x\in M_q$, then $p+x\leq p$.
\end{lem}
\begin{proof}
We separate two cases:
\begin{enumerate}
\item If $q$ is regular, then for any $x\in M_q$ we have $q+x\leq nq\leq q$. By Lemma \ref{lem:Tech1} $p\leq p+q\leq p$, and thus $p+x\leq (p+q)+x=p+(q+x)\leq p+q\leq p$.
\item If $q$ is free, then by Lemma \ref{lem:archcomp} there exist $n\in \N$, $a\in M$ with $q+a\leq q$ and $x=nq+a$. Again by Lemma \ref{lem:Tech1}, $ p+q\leq p$, and thus $p+x=p+nq+a\leq p+nq\leq p$.
\end{enumerate}
\end{proof} 

For $x\in M_q$, set
$$\varphi _{pq}(x) = (p+x) -p\in G_p\subseteq G(M_p).$$
Observe that this map is well-defined by Lemma \ref{lem:elkefaltava}. Also, it is clearly a semigroup homomorphism.
It is straightforward to show that the induced maps
$\widehat{\varphi}_{pq}\colon \widehat{G}_q\to \widehat{G}_p$ 
satisfy condition (c1) of Definition \ref{def:I-system}. We finally look at condition (c2). Let $p$ be a
free prime in $\mathbb P$ and let $\alpha \in M$ be a nonzero
element such that $p+\alpha \le p$. By \cite[Theorem 5.8]{Brook},
there exist pairwise incomparable primes $q_1,\dots ,q_r\in \mathbb
P$ such that $q_i<^* p$ for all $i$, and uniquely determined positive
integers $n_i$, with $n_i=1$ if $q_i$ is regular,
 such that
$$\ol{\alpha} = n_1\ol{q}_1+\dots +n_r\ol{q}_r.$$
Since $\alpha \le n_1q_1+\cdots +n_rq_r$, we can apply refinement to
get $\alpha =\alpha_1+\cdots +\alpha_r$ with $\alpha_i \le n_iq_i$
for all $i$. Observe that we obtain $\ol{\alpha_i}= n_i \ol{q}_i$,
as otherwise we would arrive to a contradiction with the uniqueness of
the expression of $\ol{\alpha}$ as a sum of primes in $\ol{M}$. In
particular we obtain $\alpha _i\in M_{q_i}$, and
$$(p+\alpha) -p = \sum\limits _{i=1}^r \varphi _{pq_i} (\alpha_i)\in (\bigoplus _{q<p}\varphi _{pq})(\bigoplus _{q<p}M_q) , $$
as desired.\vspace{.3truecm}

We have thus built a $\mathbb{P}$-system $\mathcal{J}_M = (\mathbb P, \le^*, (G_p)_{p\in \mathbb P},
\varphi _{pq} (q<^*p))$ associated to the monoid $M$.

\begin{theor}
\label{thm:M=M(J)} With the above notation, we have that there is a
natural isomorphism of monoids
$$M(\mathcal{J}_M) \longrightarrow  M. $$
\end{theor}
\begin{proof}
For each $p\in \mathbb{P}$, there is a natural map
$$\psi_p:M({\mathcal{J}_M})_p\rightarrow M_p,$$
which coincides with the identity map when $p$ is regular, and with the map $(n, (p+a)-p)\mapsto np+a$ when $p$ is free and $p+a\in G'_p$. 
As observed in Remark \ref{rem:psi-iso}, the map $\psi _p$ is an isomorphism.
Now, let $q< ^*p\in \mathbb{P}$, let $x\in M_p$ and $y\in M_q$. By Lemma \ref{lem:elkefaltava}, we have $x+y\in M_p$, and so $x+y= x+y+0_p
= x+ [(p+y)-p]=x+\varphi_{pq}(y)$ holds in $M$, where $0_p$ is the identity of the group $\widehat{G}_p$. By Corollary 
\ref{cor:presentationofM}, there exists  a unique monoid homomorphism $\psi \colon
M(\mathcal{J}_M)\to M$ which restricts to $\psi_p$ for every $p\in \mathbb{P}$. Since $M$ is primely generated and conical, $\psi$ is a surjective map.

To show that $\psi $ is injective, we use the description of
$M(\mathcal{J}_M)$ obtained above and an argument of Dobbertin
\cite[pp. 172--173]{Dobb84}. Since $\psi_p$ is an isomorphism for all $p\in \mathbb P$, and in order to simplify the notation, we will identify 
$M(\mathcal J _M)_p$ and $M_p$ for the rest of the proof.

Let $\tilde{x},\tilde{y}\in M(\mathcal{J}_M)$ be such that $\psi
(\tilde{x})=\psi (\tilde{y})$. Adopting the notation introduced
in Section \ref{section:Isystems}, we may assume that $\tilde{x},\tilde{y}$ have representatives
 $x\in H_a$ and $y\in H_b$, respectively, for $a,b\in A(I)$, of the form
$$x=\sum\limits _{p\in \Ma (a)} \chi (a, p, x_p),\qquad y= \sum\limits _{q\in \Ma (b)} \chi (b,q,y_q) ,$$
where $x_p\in M_p$ for all $p\in \Ma (a)$ and $y_q\in M_q$ for all
$q\in \Ma (b)$.

Observe that
$$\sum\limits _{p\in \Ma (a)} x_p= \psi (\tilde{x}) = \psi (\tilde{y}) =
\sum\limits _{q\in \Ma (b)} y_q$$ in $M$. Applying \cite[Theorem
5.8]{Brook}, we deduce that $\Ma (a)=\Ma (b)$, and so $a=b$.

Since $M$ has refinement, there are elements $z_{pq}$ in $M$ such
that $x_p= \sum\limits _{q\in \Ma (b)} z_{pq}$ and $y_q= \sum\limits _{p\in \Ma
(a)} z_{pq}$. Since $M$ is primely generated, we can write each $z_{pq}$ as a finite sum
$$z_{pq}= \sum\limits_l  w_{pql} ,$$
where $w_{pql}\in M_l$ for $l\in \mathbb P$. Observe that, looking
at $x_p$ as an element of $\widehat{G}_p$, we have  $x_p= \sum\limits _q(\sum\limits
_l \widehat{\varphi} _{pl}(w_{pql}))$, and similarly, looking at $y_q$
as an element of $\widehat{G}_q$, we have  $y_q = \sum\limits _p (\sum\limits _l
\widehat{\varphi} _{ql} (w_{pql}))$.

 The following computation is performed in the group $\widehat{H}_a$
(cf. \cite{Dobb84}):
\begin{align*}
  & \sum\limits _{p,q,l} \Big(\chi (a,p, \widehat{\varphi}
_{pl}(w_{pql}))-\chi(a,l,w_{pql})\Big) \\
&  + \sum\limits_{p,q,l} \Big( \chi
(a,l,w_{pql})- \chi (a,q,\widehat{\varphi}_{ql}(w_{pql}))\Big) \\
= &   \sum\limits _{p,q,l} ( \chi(a,p,\widehat{\varphi}_{pl}(w_{pql}))- \chi
(a,q,\widehat{\varphi} _{ql}(w_{pql})) ) \\
& = \sum\limits _p\Big( \chi (a, p,\,  \sum\limits _{q}(\sum\limits _l \widehat{\varphi}
_{pl}(w_{pql}))) \Big) - \sum\limits _q \Big( \chi (a,q, \, \sum\limits _{p}(\sum\limits
_l \widehat{\varphi}
_{ql}(w_{pql})))\Big) \\
& = \Big( \sum\limits _p   \chi (a,p, x_p )\Big) - \Big( \sum\limits _q  \chi
(a,q,
y_q) \Big) \\
& =  x-y.
\end{align*}
This shows that $x-y\in U_a$, and so $\tilde{x}=\tilde{y}$ in $M_a= H_a/\sim $ (see Corollary  \ref{cor:NouSubmonoid}). This concludes the proof.
\end{proof}

\begin{rema}
 \label{remark:primes}
{\rm It can be easily checked that, for any $I$-system $\mathcal J$,  the primes of $M(\mathcal J)$ are precisely the elements
 of the form $\chi_i (x)$, for $i\in I_{{\rm reg}}$ and $x\in G_i$ and the elements of the form $\chi _j((1,x))$ for 
 $j\in I_{{\rm free}}$ and $x\in G_j$. So, $M(\mathcal J)$ is always a primely generated monoid.}
 \end{rema}

\begin{prop}\label{prop:finitely-generated}
An $I$-system $\mathcal J$ is finitely generated if and only if $M(\mathcal J)$ is a finitely generated monoid.
\end{prop}
\begin{proof}
Assume first that $\mathcal J= (I,\le , (G_i)_{i\in I}, \varphi_{ji} (i<j))$ is a finitely generated $I$-system. By Remark \ref{rem:minimals-in-I}, for 
every $i\in I_{free}$ minimal, the semigroup $M_i$ is generated by $(1, e_i)$, and thus it is finitely generated. Then, using condition (c2) in Definition \ref{def:I-system}, it is easily seen that $M(\mathcal J)$ is generated as a monoid by the elements of the form $\chi _i(1,e_i)$, for $i\in I_{{\rm free}}$ and the elements of the form $\chi_i(  x_{i,t})$, for $i\in I_{{\rm reg}}$, where $\{ x_{i,1},\dots ,x_{i,l_i} \}$ is a finite family of semigroup generators of $G_i$.
 
Conversely, suppose that $M(\mathcal J)$ is finitely generated. We first show that $I$ is finite. Indeed consider the $I$-system $\mathcal J^{\equiv}$, with the same partition of $I$ as disjoint union of free and regular elements, and with $G^{\equiv}_i=\{ e_i \}$ for all $i\in I$. Then $M(\mathcal J^{\equiv})$ is the antisymmetrization of $M(\mathcal J)$, and so is a finitely generated monoid. It is readily seen that a minimal set of generators of the primely generated  refinement monoid $M(\mathcal J^{\equiv})$ \cite{Pierce} is precisely the set of primes $\mathbb P (M(\mathcal J^{\equiv}) )= I$ of $M(\mathcal J^{\equiv})$. Therefore $I$ is finite. 

 Now, let $i\in I_{{\rm reg}}$, and let $\mathcal G $ be a finite family of semigroup generators of $M(\mathcal J)$. If $g\in \mathcal G$ and 
 $\chi_i(e_i)+g\le \chi_i(e_i)$, then there is an element $g_i\in G_i$ such that $\chi_i(g_i)= \chi _i(e_i)+g$. It is easy to check that the finite family
$$\{ g_i : g\in \mathcal G \text{ and } \chi_i(e_i) +g\le \chi_i(e_i) \}    $$
 generates $G_i$ as a semigroup. Therefore, $G_i$ is a finitely generated group if $i\in I_{{\rm reg}}$. 
 Now, if $i\in I_{{\rm free}}$ then one shows using (c2) in Definition \ref{def:I-system} and induction that $G_i$ is also a finitely generated abelian group.
 \end{proof}

\section{The refinement property of $M(\mathcal J)$ for finitely generated $I$-systems}
\label{sect:ref-for-fg}

In this section we show that, for any finitely generated $I$-system $\mathcal J$, the monoid $M(\mathcal J)$ has the refinement property.
Indeed, we prove more generally this result for arbitrary $I$-systems over finite posets $I$. This will be used in the next section to show the refinement property for 
the monoids associated to arbitrary $I$-systems. We remark that a main difficulty in establishing the refinement property for $M(\mathcal J)$ is that the components $M_a$, for $a\in A(I)$, do not satisfy refinement in general. 

The first step is to show that the result
holds when $I$ satisfies the additional condition that, for
every $p\in I$, the set $ I \uparrow p=\{ i\in I : p\le i\}$ is a
chain (Proposition \ref{prop:chins-up}). For this result, we do not require $I$ to be finite. Given $i\in I$, we will denote  the lower subset $I \downarrow i=\{x\in I : x\leq i\}$ by $a(i)$. 
\vspace{.2truecm}

We establish in the next result one of the crucial steps for proving the refinement property.

\begin{lem}
\label{lem:refinement-down} Let $\mathcal J = (I, \le , G_i,
\varphi_{ji}(i<j))$ be an $I$-system.
Let $i\in I$ and consider elements
$x_1,x_2,y_1,y_2\in M_i$ such that $x_1+x_2=y_1+y_2$ in $M_i$. 
Then there are elements $z^{(rs)}$ such that $z^{(11)}=\chi (a(i),
i, z_{11})$, $z^{(22)}= \chi (a(i), i, z_{22})$, for some elements
$z_{11},z_{22}$ in $M_i$,  $z^{(12)}, z^{(21)}\in \bigsqcup _{b\subseteq
a(i)} H_b$ such that the identities $[\chi (a(i), i, x_r)]= \sum\limits _s
[z^{(rs)}]$ for $r=1,2$ and $[\chi (a(i), i, y_s)]=\sum\limits _r
[z^{(rs)}]$ for $s=1,2$ hold in $M(\mathcal J)$.
\end{lem}
\begin{proof}
If $i\in I_{reg}$, then $M_i$ is a group, whence the result is clear. So assume that $i\in
I_{free}$. Write $x_r=(n_r,g_r)\in \mathbb N\times G_i$ and
$y_s=(m_s,h_s)\in \mathbb N\times G_i$, for $r,s\in \{1,2\}$.
Consider a refinement of the equality $n_1+n_2=m_1+m_2$ of the form
$$\begin{tabular}{|l|l|l|}
\cline{2-3} \multicolumn{1}{l|}{} & ${m_1}$ & ${m_2}$ \\ \hline
${n_1}$ & ${\alpha_{1,1}}$ & ${\alpha_{1,2}}$ \\ \hline ${n_2}$ &
${\alpha_{2,1}}$ & ${\alpha_{2,2}}$ \\ \hline
\end{tabular}$$
with $\alpha_{1,1}$ and $\alpha _{2,2}$ positive integers and
$\alpha_{1,2}$ and $\alpha _{2,1}$ non-negative integers. On the
other hand we may consider a refinement
$$\begin{tabular}{|l|l|l|}
\cline{2-3} \multicolumn{1}{l|}{} & ${h_1}$ & ${h_2}$ \\ \hline
${g_1}$ & ${\beta_{1,1}}$ & ${\beta_{1,2}}$ \\ \hline ${g_2}$ &
${\beta_{2,1}}$ & ${\beta_{2,2}}$ \\ \hline
\end{tabular}$$
of the identity $g_1+g_2= h_1+h_2$ in the group $G_i$. For the
indices $(r,s)$ such that $\alpha _{r,s}=0$, use condition (c2) in Definition \ref{def:I-system} to
find finitely many elements $\delta_k ^{(rs)}\in M_k$, with $k<i$, satisfying that $\sum\limits _{k} \varphi_{ik} (\delta ^{(rs)}_k)= \beta_{r,s}$. Now take $z_{11}=
(\alpha_{1,1}, \beta_{1,1})$, $z_{22}= (\alpha_{2,2},\beta_{2,2})$,
and for $r\ne s$, take $z^{(rs)}= \chi (a(i), i ,
(\alpha_{r,s},\beta _{r,s}))$ in case $\alpha_{r,s}\ne 0$ and
$z^{(r,s)}= \sum\limits _k \chi (a(k),k, \delta _k^{(rs)})\in H_b$ in case $\alpha_{r,s}=0$, where $b\in A(I)$ and $b\subset a(i)$. It is
then clear by Lemma \ref{lem:equivrels} that the identities $[\chi (a(i), i, x_r)]= \sum\limits _s
[z^{(rs)}]$ for $r=1,2$ and $[\chi (a(i), i, y_s)]=\sum\limits _r
[z^{(rs)}]$ for $s=1,2$ hold in $M(\mathcal J)$.
\end{proof}

Now, we are ready to prove the desired result for a special kind of posets.

\begin{prop}
\label{prop:chins-up} Let $\mathcal J = (I, \le , G_i,
\varphi_{ji}(i<j))$ be an $I$-system. Assume that, for
every $p\in I$, the set $ I \uparrow p=\{ i\in I : p\le i\}$ is a
chain. Then, the associated monoid $M(\mathcal J )$ is a refinement
monoid.
\end{prop}
\begin{proof}
Assume that $x^{(1)}+x^{(2)}=y^{(1)}+y^{(2)}$ in $M:=M(\mathcal J)$,
where $x^{(r)}\in M_{a_r}$ and $y^{(s)}\in M_{b_s}$ for $a_r,b_s\in
A(I)$.  Then $a:= a_1\cup a_2=b_1\cup b_2$. Write $A_r=\Ma (a_r)$,
$B_s= \Ma (b_s)$, and $A= \Ma (a)$. Observe that $A_1\cap
A_2\subseteq A\subseteq A_1\cup A_2$ and similarly $B_1\cap
B_2\subseteq A\subseteq B_1\cup B_2$.

We first reduce to the case where $A$ is a singleton. For this, observe
that, because of our hypothesis that $I\uparrow p$ is a chain for
every $p\in I$,  the sets $a(k)$, for $k\in A$, are
mutually disjoint. Now, for each $k\in A$, let $N_k=\bigsqcup_{b\subseteq a(k)}M_b$ be the order-ideal of $M$
generated by the archimedian component $M_{a(k)}$. By the above remark, we have that the internal direct sum of order-ideals
$ N:=\sum\limits^{\oplus} _{k\in A} N_k $
is the order ideal of $M$ generated by $x^{(1)}+x^{(2)}=y^{(1)}+y^{(2)}$. Restricting the equality $x^{(1)}+x^{(2)}=y^{(1)}+y^{(2)}$ to each
$N_k$, $k\in A$, we may thus assume that $A=\{k\}$ for a single
element $k\in I$ (and consequently $a=a(k)$). Note that there exist
$r,s\in \{ 1,2 \}$ such that $k\in A_r\cap B_s$. Without loss of
generality, we shall assume that $r=s=1$, so that $k\in A_1\cap
B_1$.

By Lemma \ref{lem:equivrels}, for $r=1,2$ we can take representatives $\tilde{x}^{(r)}$ of $x^{(r)}$ in
$H_{a_r}$ of the form $\tilde{x}^{(r)}= \sum\limits _{i\in A_r}
\tilde{x}^{(r)}_i$, where $\tilde{x}^{(r)}_i\in M_i$ for $i\in A_r$
(and where $\tilde{x}^{(r)}_i=0_i\in \widehat{G}_i$ if $i\in
a_r\setminus A_r$). Similarly, we can  take representatives $\tilde{y}^{(s)}$
of $y^{(s)}$ in $H_{b_s}$ of the form $\tilde{y}^{(s)}= \sum\limits _{i\in
B_s} \tilde{y}^{(s)}_i$, where $\tilde{y}^{(s)}_i\in M_i$ for $i\in
B_s$.

Since $\tilde{x}^{(1)}+\tilde{x}^{(2)}\sim
\tilde{y}^{(1)}+\tilde{y}^{(2)}$ in $H_a$, it follows that there are
elements $u_i\in \widehat{G}_i$ for $i<k$ such that
\begin{equation}
\label{eq:tilderels}
 \tilde{x}^{(1)}+\tilde{x}^{(2)}-(
\tilde{y}^{(1)}+\tilde{y}^{(2)}) =\sum\limits _{i<k} \Big( \chi
(a,i,u_i)-\chi (a,k, \widehat{\varphi}_{ki}(u_i))\Big) .
\end{equation}  For $i\in a \setminus ( A_1\cup A_2\cup B_1\cup B_2)$ we
thus obtain that $u_i=0_i$.

We now proceed to obtain the refinement. We need to distinguish
several cases. To start with, observe that the refinement is trivial
in case $x^{(2)}=0$ or $y^{(2)}=0$, so we will assume that $A_2\ne
\emptyset $ and $B_2 \ne \emptyset$. 

Since for any $i=1,2$ we have $a_i\subseteq a_1\cup a_2=a(k)$ and $b_i\subseteq b_1\cup b_2=a(k)$, we have that $y\leq k$ for every $y\in A_i$ and for every $y\in B_i$ ($i=1,2$). 

Assume first that $k\in A_1\cap A_2\cap B_1\cap B_2$. Then we have $A_1=A_2=B_1=B_2=\{
k\}$. So, we obtain from (\ref{eq:tilderels}) that
$\tilde{x}_k^{(1)}+\tilde{x}_k^{(2)}=\tilde{y}^{(1)}_k+\tilde{y}^{(2)}_k$ in $M_k$, and the result
follows from Lemma \ref{lem:refinement-down}.

A second case appears when $k\in A_1\cap A_2\cap B_1$ but $k\notin
B_2$. Then we have $A_1=A_2=B_1=\{k\}$, and $b_2=\bigcup\limits_{d\in B_2}a(d)\subsetneq a$. From (\ref{eq:tilderels}), we get
\begin{equation}
\label{eq:K(B_s)} \tilde{x}_k^{(1)}+\tilde{x}^{(2)}_k
-\tilde{y}_k^{(1)}= - \sum\limits _{d\in B_2}
\widehat{\varphi}_{kd}(u_d),\qquad \tilde{y}^{(2)}_d= - u_d \quad (d\in
B_2).
\end{equation} Set
\begin{equation}
\label{eq:z8} z^{(22)}= \tilde{y}^{(2)}, \quad z^{(12)}=0,\quad
z^{(11)}=\tilde{x}^{(1)},\quad  z^{(21)}= \chi (a, k,
\tilde{x}_k^{(2)}- \sum\limits _{d\in B_2} \widehat{\varphi} _{kd}
(\tilde{y}^{(2)}_d )).
\end{equation}
Then it follows from (\ref{eq:K(B_s)}) that $x^{(r)}= \sum\limits _s
[z^{(rs)}]$ and $y^{(s)}= \sum\limits_r [z^{(rs)}]$ in $M(\mathcal J)$,
giving the desired refinement. The case where $k\in A_1\cap B_1\cap B_2 $ and $k\notin A_2$ is
treated similarly.

Finally we consider the case where $k\notin A_2\cup B_2$. (Recall
that we are assuming throughout that $k\in A_1\cap B_1$). Then we have $A_1=B_1=\{k\}$ and $a_2, b_2\subsetneq a$. We have
from (\ref{eq:tilderels}):
\begin{align}
\label{eq:usandxys} \tilde{x}_k^{(1)}& -\tilde{y}_k^{(1)}=- \sum\limits
_{d\in A_2\cup B_2} \widehat{\varphi}_{kd} (u_d), \qquad u_d=x^{(2)}_d
\,
\text{ for } d\in A_2\setminus B_2,\\
\label{eq:usandxys2}
 u_d= & -y^{(2)}_d \, \text{ for } d\in B_2\setminus A_2, \qquad
 u_d=  x^{(2)}_d-y^{(2)}_d \, \text{ for } d\in A_2\cap B_2,
\end{align}

Set $z^{(22)}=0$, $z^{(12)}= \tilde{y}^{(2)}=\sum\limits _{d\in B_2}
\tilde{y}^{(2)}_d$, $z^{(21)}= \tilde{x}^{(2)}=\sum\limits _{d\in A_2}
\tilde{x}^{(2)}_d$, and
$$z^{(11)}= \chi \Big( a(k), k, \tilde{x}^{(1)}_k - \sum\limits _{d\in B_2}
\widehat{\varphi}_{kd} (\tilde{y}^{(2)}_d)\Big). $$ It is clear that
$x^{(2)}= [z^{(21)}]+ [z^{(22)}]$, $y^{(2)}= [z^{(12)}]+[z^{(22)}]$
and $x^{(1)}= [z^{(11)}]+[z^{(12)}]$. Using equations
(\ref{eq:usandxys}) and (\ref{eq:usandxys2}), we also obtain
$y^{(1)}= [z^{(11)}]+[z^{(21)}]$. This concludes the proof.
\end{proof}

Now we start our approach to the proof of the refinement property for the monoids associated to systems over finite posets. We first analyze the functoriality of our main construction. 

\begin{defi}\label{def:Funct1}
{\rm Let $I^{(t)}$ be posets ($t=1,2$), and let $\mathcal J_t= (I^{(t)}, \le, G_i^{(t)}, \varphi ^{(t)}_{ji} (i<j) )$ be $I^{(t)}$-systems. A homomorphism of systems $f\colon \mathcal J _1\to \mathcal J _2$ consists of an order-preserving map
$\psi\colon I_1\to I_2$ such that $\psi (i) <\psi (j)$ whenever $i<j$ (that is, it is injective on chains),  
$\psi (I^{(1)}_{{\rm free}} )\subseteq I^{(2)}_{{\rm free}}$ and 
$\psi (I^{(1)}_{{\rm reg}} )\subseteq I^{(2)}_{{\rm reg}}$,
and a family of group homomorphisms $f_i \colon G^{(1)}_i \to G^{(2)}_{\psi (i)} $ such that for $i<j$ in $I^{(1)}$ the following diagram is commutative:
\begin{equation}\label{diagram:camell}
\begin{CD}
M_i^{(1)} @>\varphi_{ji}^{(1)}>> G_j^{(1)}\\
   @V\ol{f}_i VV  @VVf_j V\\
M^{(2)}_{\psi (i)} @>\varphi_{\psi(j)\psi(i)}^{(2)}>> G_{\psi (j)}^{(2)}
\end{CD}
 \end{equation}
where for $i\in I^{(1)}_{{\rm reg}}$, $\ol{f}_i =  f_i$, and,  for $i\in I^{(1)}_{{\rm free}}$, $\ol{f}_i\colon M_i^{(1)}\to M^{(2)}_{\psi (i)}$ is defined by
$\ol{f}_i (n,h)= (n,f_i (h))$, for $n\in \N$ and $h\in G^{(1)}_i$.} 
\end{defi}

We have the following result:

\begin{lem}\label{lem:Funct2}
Let $I^{(t)}$ be posets ($t=1,2$), and let $\mathcal J_t$ be $I^{(t)}$-systems. Then, any homomorphism of systems $f\colon \mathcal J _1\to \mathcal J _2$ 
induces a monoid homomorphism $M(f)\colon M(\mathcal J _1)\to M(\mathcal J_2)$.
\end{lem}
\begin{proof}
We set $M(f) (\chi _i(x)) = \chi _{\psi (i)}(\ol{f}_i(x))$ for all $x\in M^{(1)}_i$, $i\in I^{(1)}$. 
By Corollary \ref{cor:presentationofM}, to show that this is a well-defined homomorphism, it is enough to show that, if $i<j$ in $I^{(1)}$,
$x\in M^{(1)}_j$ and $y\in M^{(1)}_i$, then
$$\chi_{\psi(j)}(\ol{f}_j (x)) + \chi_{\psi (j)}(\ol{f}_i (y)) = \chi_{\psi (j)}(\ol{f}_j( x+\varphi_{ji}^{(1)} (y))) .$$
For this, observe that
\begin{align*}
 & \chi_{\psi (j)}(\ol{f}_j  (x+\varphi _{ji}^{(1)} (y)))  = \chi_{\psi (j)}(\ol{f}_j (x) + f_j(\varphi ^{(1)}_{ji} (y))) \\
 & = \chi_{\psi (j)}(\ol{f}_j(x) + \varphi^{(2)}_{\psi (j)\psi(i)} (\ol{f}_i (y))) 
= \chi_{\psi (j)}(\ol{f}_j (x)) + \chi _{\psi (i)}(\ol{f}_i (y))\, ,
\end{align*}
where we have used the commutativity of the diagram (\ref{diagram:camell}) for the second equality.
\end{proof}

Recall that given a poset $I$, and an element $i\in I$, the {\it lower cover} of $i$ in $I$ is the set 
$$\rL(I,i)=\{j\in I : j<i \text{ and } [j,i]=\{j,i\}\}.$$ 
Under certain circumstances we can pullback an $I$-system, as follows.

\begin{lem}
 \label{lem:pullbacking-system}
Let $I^{(1)}, I^{(2)}$ be finite posets, let $\mathcal J _2= (I^{(2)},\le , G^{(2)}_i, \varphi^{(2)}_{ji} (i<j)) $ be an $I^{(2)}$-system, 
and let $\psi \colon I^{(1)}\to I^{(2)}$ be an order-preserving surjective map such that $\psi (i) < \psi (j)$ for $i<j$ in $I^{(1)}$.
 Assume moreover that $\psi $ induces a bijection from $\rL (I^{(1)}, i)$ onto $\rL (I^{(2)}, \psi (i) )$ for all $i\in I^{(1)}$.
 Set $G^{(1)}_i= G^{(2)}_{\psi (i)}$ for all $i\in I^{(1)}$, and $\varphi ^{(1)}_{ji} = \varphi ^{(2)} _{\psi (j) \psi (i)}$ 
 for all $i,j\in I^{(1)}$ with $i<j$. Then, $\mathcal J _1 = (I^{(1)},\le , G^{(1)}_i, \varphi^{(1)}_{ji} (i<j) )$ is an $I^{(1)}$-system, and there is a natural 
 homomorphism of systems $f\colon \mathcal J _1 \to \mathcal J _2$ inducing a surjective monoid homomorphism
 $M(f)\colon M(\mathcal J _1) \to M(\mathcal J _2)$.
 \end{lem}
\begin{proof}
 The proof is straightforward. The only thing to be remarked is that property (c2) for the $I^{(1)}$-system $\mathcal J _1$ follows from the 
 condition that  $\psi $ induces a bijection from $\rL (I^{(1)}, i)$ to $\rL (I^{(2)}, \psi (i) )$ for all $i\in I^{(1)}$. Indeed, assume that $i\in I^{(1)}_{{\rm free}}$.
 It suffices to note that, given any $j'\in I^{(2)}$ such that $j'<\psi (i)$, there exists $j\in I^{(1)}$ such that $j<i$ and $\psi (j)= j'$. 
 For this, take a chain $j'= j_0'<j_1'\cdots < j_l'= \psi (i)$ in $I^{(2)}$ such that $j_t'\in \rL(I^{(2)}, j_{t+1}')$ for $t=0,\dots , l-1$. Then, by using our hypothesis,
 we can build a sequence $j_0<j_1 \cdots <  j_l=i$ in $I^{(1)}$ such that $j_t\in \rL(I^{(1)}, j_{t+1})$ for $t=0,\dots , l-1$ and $\psi (j_t)=j_t'$ for all $t$.
 Now set $j=j_0$. 
\end{proof}

To obtain the refinement of $M(\mathcal J)$ for a general $I$-system $\mathcal J$ over a finite poset $I$, we
will use a technique introduced in \cite[Section 6]{Aposet}. In that paper,
given a finite poset $\mathbb P$ with a greatest element, another poset $\mathbb F$ is
constructed with the property that $\mathbb F \uparrow i$ is a
chain for every $i\in \mathbb F$, such that there is a surjective
order-preserving map $\Psi \colon \mathbb F\to \mathbb P$ satisfying
certain properties \cite[Proposition 6.1]{Aposet}. Denoting by $M(\mathbb P)$ the monoid generated
by $\mathbb P$ with the only relations given by the rules $p+q=p$
whenever $q<p$, it was shown in \cite[Proposition 6.5]{Aposet} that $M(\mathbb P)$ is obtained
from $M(\mathbb F)$ by a sequence of crowned pushouts (see below for
the definition). We aim here to obtain a corresponding result for
the monoids $M(\mathcal J)$, which in particular will provide a
proof of the refinement property for them.

We will use here \cite[Proposition 6.1]{Aposet} and the order-theoretic method behind the proof
of \cite[Proposition 6.5]{Aposet}. The monoid content of \cite[Proposition 6.5]{Aposet} 
needs to be adapted in order to be applied to our situation. We proceed to do that adaptation, 
in various steps.\vspace{.2truecm}

Let us recall from \cite[Section 4]{Aposet} the definition of a
crowned pushout.

\begin{defi}
\label{def:crownedpushout}
{\rm  Let $P$ be a conical monoid.  Suppose
that $P$ contains order-ideals $I$ and $I'$, with $I\cap I'=0$, such
that there is an isomorphism $\varphi \colon I\to I'$. We have a
diagram
$$\begin{CD}
I @>=>> I\\
   @V\varphi VV  @VV\iota_1 V\\
I' @>\iota_2>> P
\end{CD}
$$
which is not commutative.

The {\it crowned pushout} $Q$ of $(P,I,I',\varphi)$ is the coequalizer of the maps $\iota_1\colon I\to P$ and $\iota
_2\circ \varphi\colon I\to P$, so that there is a map $f\colon P\to
Q$ with $f(\iota _1(x))=f(\iota _2(\varphi (x)))$ for all $x\in I$
and given any other map $g\colon P\to Q'$ such that $g(\iota
_1(x))=g(\iota _2(\varphi (x)))$ for all $x\in I$, we have that $g$
factors uniquely through $f$.}
\end{defi}

\begin{prop}[{\cite[Proposition 4.2]{Aposet}}]
\label{crownpushmon} Let $P$ be a conical refinement monoid. Suppose
that $P$ contains order-ideals $I$ and $I'$, with $I\cap I'=0$, such
that there is an isomorphism $\varphi \colon I\to I'$. Let $Q$ be
the crowned pushout of $(P,I,I', \varphi)$. Then, $Q$ is the monoid
$P/\sim$ where $\sim$ is the congruence on $P$ generated by $x+i\sim
x+\varphi (i)$ for $i\in I$ and $x\in P$. Moreover $Q$ is a conical
refinement monoid, and $Q$ contains an order-ideal $Z$, isomorphic
with $I$, such that the projection map $\pi \colon P\to Q$ induces
an isomorphism $P/(I+I')\cong Q/Z$.
\end{prop}

It was observed in the proof of \cite[Proposition 4.2]{Aposet} that
the equivalence relation $\sim $ on $P$ is {\it refining}, that is,
if $x\sim y+z$ then there is a decomposition $x=x_1+x_2$ such that
$x_1\sim y $ and $x_2\sim z$. In the terminology of \cite{Dobb84},
this means that the quotient map $\pi \colon P \to Q$ is a
V-homomorphism. The refinement of $Q$ follows from this fact.\vspace{.3truecm}

\begin{point}\label{pt:KeyPoint}
{\rm Let $I$ be a finite poset, let $\mathcal J= (I, \le , G_i, \varphi _{ji} (i<j))$ be an $I$-system, 
and let $i$ be a maximal element of $I$. Since $I\downarrow i$ is a finite poset, we can take the $(I\dnw i)$-system
$$\mathcal J_i=(I\downarrow i, \leq , (G_j)_{j\in I\downarrow i}, \varphi_{jk} (k<j, j\in I\downarrow i))$$ 
obtained by restricting $\mathcal J$ to $I\dnw i$.
Since $i$ is the greatest element for the poset $I\dnw i $, we can use the
construction in \cite[Proposition 6.1]{Aposet} to obtain a poset
$\mathbb F (i)$ and a surjective order-preserving map $\psi\colon
\mathbb F (i)\to  I\dnw i$ satisfying the following properties:
\begin{enumerate}
\item  The map $\psi $ preserves
chains, that is, if $S$ is a chain in $\mathbb F (i)$ then $\psi $
restricts to a bijection from $S$ to $\psi (S)$. 
\item The map $S\mapsto \psi (S)$ is a bijection from the
set of maximal chains of $\mathbb F(i)$ onto the set $\mathcal
S^0(i)$ of maximal chains of $I\dnw i$.
\item  For every $t\in \mathbb F(i)$, the interval $[t,i]$ is a chain \cite[Proposition 6.1]{Aposet}.
\item  For $t_1,t_2\in \mathbb F(i)$, if $\psi ([t_1,i])=\psi
([t_2,i])$ then $t_1=t_2$. (This follows directly from the construction of $\mathbb F(i)$). 
\item For every $q\in \mathbb F(i)$, the map $\psi$ induces a bijection
from $\rL(\mathbb F(i), q)$ onto $\rL(I\dnw i,\psi (q))$ \cite[Lemma 6.4]{Aposet}.
\end{enumerate}}
\end{point}

\begin{defi}
 \label{def:comp-ideals}
{\rm   Let $I$ be a finite poset and let $\mathcal J = (I, \le , G_i, \varphi _{ji} (i<j))$ be an $I$-system.
  A {\it $\mathcal J$-compatible pair} of $I$ is a pair of lower subsets  $I_1$ and $I_2$  of $I$
  such that $I_1\cap I_2=\emptyset $, and such that there is an isomorphism 
of posets $\psi \colon I_1\to I_2$ satisfying the following conditions:
\begin{enumerate}
 \item $\psi ((I_1)_{{\rm reg}}) = (I_2)_{{\rm reg}}$ and $\psi ((I_1)_{{\rm free}})= (I_2)_{{\rm free}}$.
 \item $G_i=G_{\psi (i)}$ for all $i\in I_1$.
 \item $\varphi_{ji} =\varphi _{\psi (j)\psi (i)}$ if $i<j$ and $i,j\in I_1$.
 \item $\varphi _{ji} = \varphi _{j\psi (i)}$ if $i\in I_1$, $j\in I\setminus (I_1\cup I_2)$, $i<j$, and $\psi (i) < j$.
 \end{enumerate}}
\end{defi}

We have the following easy fact.

\begin{lem}\label{lem:CompaIdeals}
Let $(I_1,I_2)$ be a $\mathcal J$-compatible pair of lower subsets, and for $t=1,2$, let $\mathcal J_t$ be the $I_t$-system obtained by restricting the $I$-system $\mathcal J$ to $I_t$.  Then, $M(\mathcal J_t )$ are order-ideals of $M(\mathcal J)$  for $t=1,2$, with
$M(\mathcal J_1)\cap M(\mathcal J _2) = 0$, and there is a monoid isomorphism
$\psi \colon M(\mathcal J _1)\to M(\mathcal J _2)$ sending $\chi_i(x)$ to $\chi_{\psi (i)}(x) $ for all $i\in I_1$ and
all $x\in M_i$.
\end{lem}
\begin{proof}
By Proposition \ref{prop:characideals}, $M(\mathcal J _t)$ are order-ideals of $M(\mathcal J)$ and, since $I_1\cap I_2=\emptyset$, we have that $M(\mathcal J_1)\cap M(\mathcal J_2)=\{ 0\}$. By Lemma \ref{lem:Funct2}, the poset isomorphism $\psi \colon I_1\to I_2$ induces a monoid isomorphism $\psi \colon M(\mathcal J _1)\to M(\mathcal J _2)$ with the desired properties.
\end{proof}

Using a compatible pair of $I$, we can construct a new system, as follows.

\begin{defi}\label{def:crownsyst}
{\rm Let $I$ be a finite poset, let $\mathcal J = (I, \le , G_i, \varphi _{ji} (i<j))$ be an $I$-system, 
 and let $(I_1,I_2)$ be a $\mathcal J $-compatible pair of lower subsets of $I$. We define $\mathcal  J ' = (I', \le ', G_i', \varphi '_{ji} (i<j) )$, where:
\begin{enumerate}
\item $I'=I\setminus I_2$.
\item $\le '$ is the order relation obtained 
by setting $i \le ' j$ if and only if either $i\le j$ in $I$ or $i\in I_1$, $j\in I\setminus (I_1\cup I_2)$, and $\psi (i) <j$ in $I$.
\item  $G_i'= G_i$ for $i\in I'$.
\item  For $i<' j$ in $I'$, we have $\varphi '_{ji}=\varphi _{ji}$ if $i< j$ in $I$, and $\varphi _{ji}'= \varphi _{j\psi (i)}$
if $i\in I_1$, $j\in I\setminus (I_1\cup I_2)$ and $\psi (i) <j$ in $I$.
\end{enumerate}}
\end{defi}

\begin{lem}\label{lem:crownsyst}
Let $I$ be a finite poset, let $\mathcal J = (I, \le , G_i, \varphi _{ji} (i<j))$ be an $I$-system, 
 and let $(I_1,I_2)$ be a $\mathcal J $-compatible pair of lower subsets of $I$. Then, $\mathcal  J ' = (I', \le ', G_i', \varphi '_{ji} (i<j) )$ is an $I'$-system.
\end{lem}
\begin{proof}
Note that condition (4) in Definition \ref{def:comp-ideals} says that $\varphi '_{ji}$ is well-defined for $i<'j$. 

To show condition (c1) in Definition \ref{def:I-system}, take $i,j,k $ in $I'$ such that $i<'j<'k$. We have to show that
$\widehat{\varphi}'_{ki} = \widehat{\varphi}'_{kj}\widehat{\varphi}_{ji}'$. 

If $i<j<k$ in $I$, then the condition follows from the corresponding condition for $\mathcal J$. If $i,j\in I_1$, $k\in I\setminus (I_1\cup I_2)$, 
and $\psi (j) <k$, then $\psi (i)< \psi (j)$ because $\psi$ is order-preserving, and thus $\psi (i) <k$. Therefore,
$$\widehat{\varphi}'_{ki}= \widehat{\varphi} _{k \psi (i)} = \widehat{\varphi} _{k\psi (j)} \widehat{\varphi}_{\psi (j)\psi (i)}= \widehat{\varphi}'_{kj}\widehat{\varphi}_{ji}=\widehat{\varphi}'_{kj}\widehat{\varphi}_{ji}' ,$$
where we have used condition (3) in Definition \ref{def:comp-ideals} for the third equality. A similar proof applies in the case where 
$i\in I_1$, $j\in I\setminus (I_1\cup I_2)$ and $\psi (i) <j$.

We show now condition (c2) in Definition \ref{def:I-system}. For this observe that, given $i\in I'$, we have
$$\bigoplus_{\{ k\in I': k<'i\} } \varphi_{ik}' =\Big( \bigoplus _{\{ k\in I': k<i \}}\varphi_{ki}\Big) \oplus \Big(\bigoplus_{\{ k\in I_1: \psi (k)<i \,\text{and}\,  k\nleq i\}} \varphi_{i \psi (k)} \Big) .$$
Since $\varphi _{i,\psi (k)}=\varphi _{ik}$ for all  $k\in I_1$ such that  $k<i$ and $\psi (k) <i$, we obtain that 
$$ \bigoplus_{\{ k\in I': k<'i\} } \varphi_{ik}'\colon \Big(\bigoplus _{\{ k\in I': k<' i \}} M_{ik}'\Big) \longrightarrow G'_i $$
is surjective.
\end{proof}

The following result plays a central role in the proof of the main result of this section.

\begin{lem}
 \label{lem:pushouting-I-systems}
 Let $I$ be a finite poset, let $\mathcal J = (I, \le , G_i, \varphi _{ji} (i<j))$ be an $I$-system, 
 and let $(I_1,I_2)$ be a $\mathcal J $-compatible pair of lower subsets of $I$. 
 Then, the crowned pushout of $M(\mathcal J )$ with respect to $\psi \colon M(\mathcal J _1) \to M(\mathcal J _2)$ is the
monoid $M(\mathcal J ')$, where $\mathcal  J ' $ is the system introduced in Definition \ref{def:crownsyst}.
\end{lem}
\begin{proof}
Let $Q$ be the crowned pushout of $M(\mathcal J )$ with respect to $\psi \colon M(\mathcal J _1) \to M(\mathcal J _2)$, and let $f\colon M(\mathcal J)\to Q$ be the canonical homomorphism.

There is a surjective homomorphism $\pi \colon M(\mathcal J) \to M(\mathcal J ')$  which sends $\chi_i (x)$ to $\chi_i (x)$ if $i\in I\setminus I_2$ and
$x\in M_i$, and sends 
$\chi_{\psi (i)} (y)$ to $\chi_i (y)$ if $i\in I_1$ and $y\in M_{\psi (i)}= M_i$. Obviously this homomorphism equalizes $\iota _1$ and
$\iota _2 \circ \psi $, so there is a unique homomorphism $\ol{\pi }\colon Q\to M(\mathcal J ')$ such that $\pi = \ol{\pi} \circ f$.  
Note that $\ol{\pi}$ is surjective.

To show that $\ol{\pi}$ is an isomorphism, we only need to build a homomorphism $\rho \colon M(\mathcal J ') \to Q$ such that $\rho \circ \ol{\pi}   = \text{Id}_{Q}$.
We define $\rho (\chi_i (x)) = f(\chi_i (x))$ for $i\in I'$ and $x\in M_i'=M_i$. We have to check that $\rho $ is well-defined. By Corollary \ref{cor:presentationofM},
it suffices to check that 
$$f (\chi_j (x)) + f( \chi_i (y)) = f(\chi_j ( x+ \varphi '_{ji} (y))) $$
for $i,j\in I'$ with $i<' j$, $x\in M_j'$ and $y\in M_i'$. 

Assume first that $i<j$ in $I$. Then,
\begin{equation*}
f   (\chi_j (x))  + f( \chi_i (y))  = f (\chi_j (x) + \chi_i (y)) 
 = f(\chi_j (x+ \varphi _{ji} (y)))   = f (\chi_j (x+ \varphi '_{ji} (y)) .
\end{equation*}
Suppose now that $i\in I_1$, $j\in I\setminus (I_1\cup I_2)$ and $\psi (i) <j$. Note that, since $f$ equalizes $\iota_1$ and $\iota_2\circ \psi$, we have
$$f(\chi_i(y)) = f(\iota_1(\chi_i(y)))=f( \iota_2\psi(\chi_i(y)))= f(\chi_{\psi (i)}(y)).$$
Hence,
\begin{align*}
& f (\chi_j (x))  + f( \chi_i (y))  = f(\chi_j (x))  + f( \chi_{\psi (i)} (y)) 
 = f (\chi_j (x) + \chi_{\psi (i)} (y)) \\
 &= f(\chi_j (x+ \varphi _{j\psi (i)} (y))) 
 = f (\chi_j (x+ \varphi '_{ji} (y))) .
\end{align*}
It is clear that  $\rho \circ \ol{\pi}   = \text{Id}_{Q}$. This concludes the proof. 
\end{proof}

The proof of the following lemma is contained in the proof of \cite[Proposition 6.5]{Aposet}.

\begin{lem}\label{lem:6.5for Posets}
Let $I$ be a finite poset and let $k\in \Ma (I)$. Then, there exist a positive integer $n$, a family $(I^t)_{0\leq t\leq n}$ of 
finite posets, and a family $\psi_t: I^t\rightarrow I^{t+1}$ of surjective order-preserving maps such that, if we denote $\Psi_{t}=\psi_{n-1}\circ \cdots \circ \psi_t$ ($0\leq t\leq n-1$), then:

\noindent {\rm (1)} $I^0=\mathbb{F}(k)$, $I^n=I\downarrow k$, and $\Psi_0$ is the map $\psi:\mathbb{F}(k)\rightarrow I\downarrow k$ in {\bf \ref{pt:KeyPoint}}.

\noindent {\rm (2)} For $t=0,1,\dots ,n-1$, there exist disjoint lower subsets $I_1^t$ and $I_2^t$ of $I^t$ such that:
\begin{enumerate}
\item[(a)] ${\Psi _{t}}|_{I_s^t}: I_s^t\rightarrow \Psi _{t} (I_s^t)$ is an isomorphism for $s= 1,2$. 
\item[(b)] $\Psi _{t}(I_1^t) =\Psi_{t} (I_2^t)$, whence one obtains an isomorphism of posets $\tau_t\colon I_1^t\to I_2^t$ by
  setting $\tau_t = (\Psi_t|_{I_2^t})^{-1}\circ (\Psi_t|_{I_1^t})$. 
  \item[(c)] $\Psi _t$ factors as $\Psi _t = \Psi _{t+1}\circ \psi_t$, where   
  $I^{t+1}= I^t\setminus I_2^t$, with the order in $I^{t+1}$ defined by $i<' j$ if either $i<j$ in $I^t$ or 
  $i\in I_1^t$, $j\in I^t \setminus (I_1^t\cup I_2^t)$ and $\tau_t (i)< j$, the map $\psi _t\colon I^t \to I^{t+1}$
  is the natural identification map, and the map $\Psi_{t+1}\colon I^{t+1}\to I\dnw k$ is the restriction of $\Psi_t$ to $I^{t+1}$.
  \end{enumerate}
\end{lem}

The following pictures illustrate the procedure described in Lemma \ref{lem:6.5for Posets} in a basic case:

$$\mathbb F = I^0  
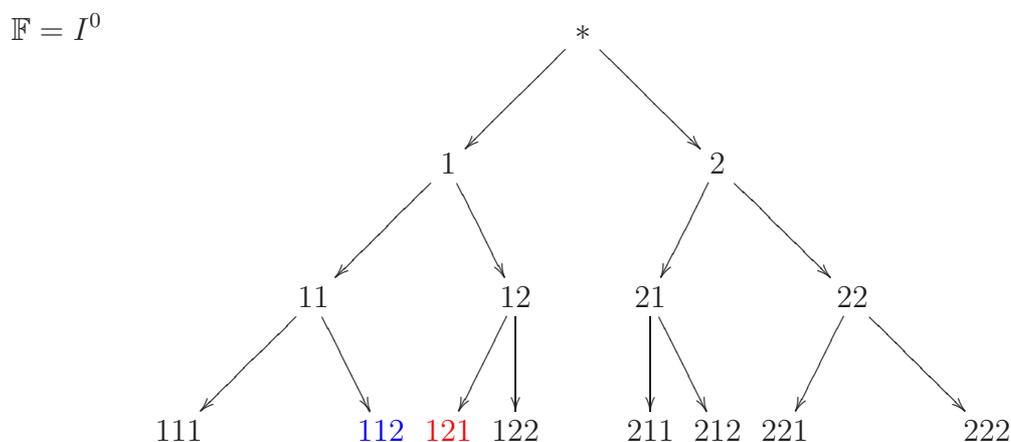
\xymatrix@!=0.1pc{ & & & & & & &  \ast \ar[ddll] \ar[ddrr]& & & & & & & \\
 & & & & & & & & & & & & & & \\
 & & &   & &1 \ar[ddll] \ar[ddr]& &  & &2  \ar[ddl] \ar[ddrr]& &   & & & \\
  & & & & & & & & & & & & & & \\
&  & & 11 \ar[ddll] \ar[ddr] & &  &12  \ar[ddl] \ar[dd] & & 21 \ar[dd] \ar[ddr] & & &  22 \ar[ddl] \ar[ddrr]& & &\\
  & & & & & & & & & & & & & & \\
  &111  & &  & \color{blue}{112} &\color{red}{121}  &122 &  & 211 &212  & 221 &  &  & 222 & 
}
$$
\vspace{1truecm}

$$I^1 
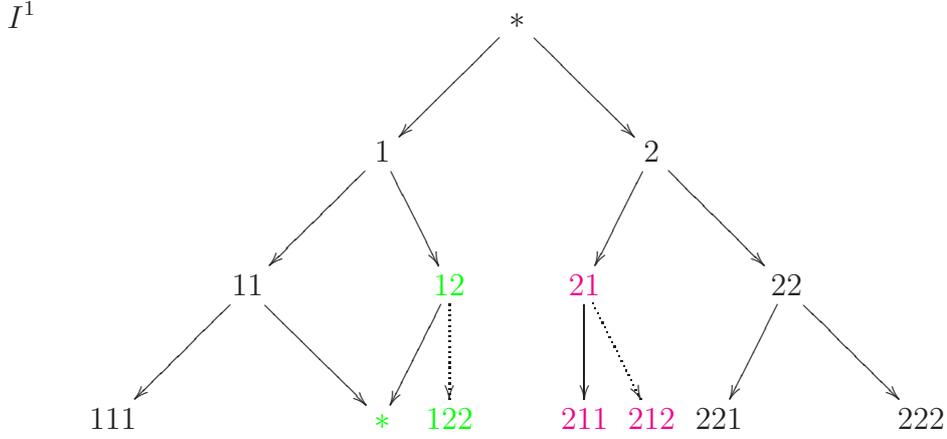
\xymatrix@!=0.1pc{ & & & & & & &  \ast \ar[ddll] \ar[ddrr]& & & & & & & \\
 & & & & & & & & & & & & & & \\
 & & &   & &1 \ar[ddll] \ar[ddr]& &  & &2  \ar[ddl] \ar[ddrr]& &   & & & \\
  & & & & & & & & & & & & & & \\
&  & & 11 \ar[ddll] \ar[ddrr] & &  &\color{green}{12}  \ar[ddl] \ar @{.>}[dd] & & \color{magenta}{21} \ar[dd] \ar @{.>}[ddr] & & &  22 \ar[ddl] \ar[ddrr]& & &\\
  & & & & & & & & & & & & & & \\
  &111  & &  &  &\color{green}{\ast}  &\color{green}{122} &  & \color{magenta}{211} &\color{magenta}{212}  & 221 &  &  & 222 & 
}
$$
\vspace{1truecm}

$$I=I^2 
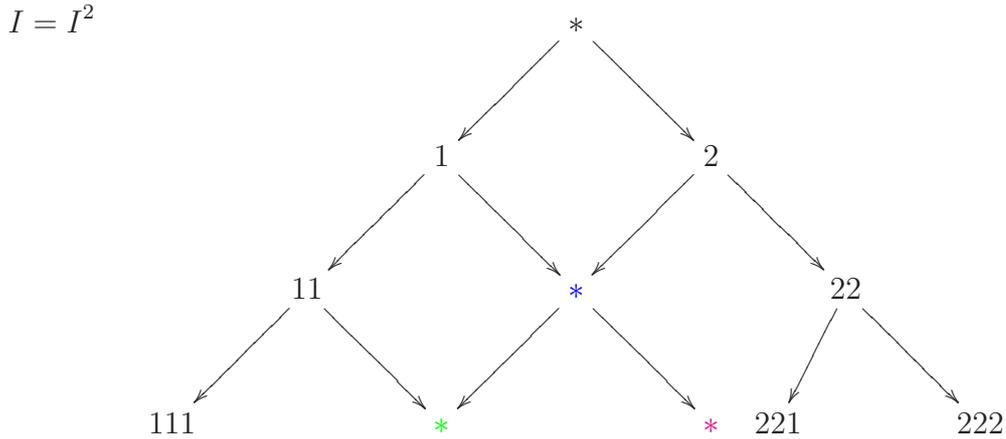
\xymatrix@!=0.1pc{ & & & & & & &  \ast \ar[ddll] \ar[ddrr]& & & & & & & \\
 & & & & & & & & & & & & & & \\
 & & &   & &1 \ar[ddll] \ar[ddrr]& &  & &2  \ar[ddll] \ar[ddrr]& &   & & & \\
  & & & & & & & & & & & & & & \\
&  & & 11 \ar[ddll] \ar[ddrr] & &  && \color{blue}{\ast}   \ar[ddll] \ar[ddrr] &  & & &  22 \ar[ddl] \ar[ddrr]& & &\\
  & & & & & & & & & & & & & & \\
  &111  & &  &  &\color{green}{\ast}  & &  &  &\color{magenta}{\ast}  & 221 &  &  & 222 & 
}
$$

In these pictures, $I$ is a poset with a  maximum element, and $\mathbb F$ is the corresponding poset having the property that 
all the subsets $\mathbb F \uparrow p$ are chains. In this case, the process described in Lemma \ref{lem:6.5for Posets} 
enables us to pass from $\mathbb F$ to $I$ in two steps. In the first step, we identify the elements $112$ and $121$ to obtain the poset $I^1$.
In the second step, we identify $(I^1)\downarrow (12)$ with $(I^1)\downarrow (21)$ to get $I^2 =I$ from $I^1$.
 
 \medskip
 
Now, we are ready to prove the main result in this section.

\begin{theor}
\label{theor:total-refinement}
Let $I$ be a finite poset and let $\mathcal J = (I,\le , G_i, \varphi _{ji} (i<j) )$ be an $I$-system. Then, $M(\mathcal J)$ is a conical refinement monoid.  
\end{theor}

\begin{proof}
We will show that $M(\mathcal{J})$ enjoys the refinement property. 

For this, we will follow the process used in the proof of \cite[Proposition 6.5]{Aposet}, using Lemmas \ref{lem:pushouting-I-systems} and \ref{lem:6.5for Posets}. Given $k\in \Ma (I)$, let $\mathbb  F (k)$ 
be the poset constructed in \cite[Proposition 6.1]{Aposet} from $I\dnw k$. By {\bf \ref{pt:KeyPoint}}, there is a surjective order-preserving map 
$\psi \colon \mathbb F (k)\to I\dnw k$ satisfying the conditions stated in Lemma \ref{lem:pullbacking-system}. Therefore, there is an $\mathbb F (k)$-system 
$\mathcal J_{\mathbb F (k)}$ and a homomorphism of systems, also denoted by $\psi$, 
$$\psi \colon \mathcal J_{\mathbb F (k)} \longrightarrow \mathcal J _k,$$
where $\mathcal J _k$ is the $(I\dnw k)$-system obtained by restricting $\mathcal J $ to $I\dnw k$, which induces a 
surjective monoid homomorphism $M(\mathcal J_{\mathbb F (k)}) \to M(\mathcal J  _k )$. Note that $M(\mathcal J  _k )$ is an order-ideal of $M(\mathcal J)$ 
(see Proposition
\ref{prop:characideals}).
  
Take the sequence of posets $I ^t$ ($t=0,1, \dots ,n$) such that $I ^0  = \mathbb F (k)$ and $I ^n= I\dnw k$, and surjective maps 
$\psi _t\colon I^t \to I^{t+1}$, for $t=0,1,\dots , n-1$, such that $\psi = \psi_{n-1}\circ \cdots \circ \psi_1\circ \psi _0$, given by Lemma \ref{lem:6.5for Posets}. 
For each $t=0,\dots ,n-1$, the map $\Psi_{t}:I^t \rightarrow I\downarrow k$ satisfies the properties required in the statement of Lemma \ref{lem:pullbacking-system}; 
this follows by induction and Lemma \ref{lem:6.5for Posets}. Indeed, assuming the result is true for $\Psi_t$, it follows from Lemma \ref{lem:6.5for Posets}(2c) that $i<' j $
in $I^{t+1}$ implies that $\Psi_{t+1}(i)< \Psi _{t+1} (j)$ in $I \dnw k$. On the other hand, if $i\in I^{t+1} = I^t\setminus I^t_2$ then, by induction, $\Psi _t$ induces a bijection from 
$\rL(I^t, i)$ onto $\rL(I\dnw k,\Psi _t(i))$. Now, if $j\in \rL(I^t,i)\cap I_2^t$, then there is a unique $j'\in I^t_1$ such that $\tau _t(j')=j$. Then, we have $j'<' i$ in $I^{t+1}$ (see Lemma  \ref{lem:6.5for Posets}(2c)), and
it easily follows that $j'\in \rL (I^{t+1}, i)$. Note that $j'\notin \rL(I^t,i)$, because $\Psi _t(j')=\Psi_t (j)$ and $\Psi_t$ is injective on $\rL(I^t,i)$. It follows that
$$\rL(I^{t+1},i)= \Big( \rL(I^t,i)\setminus I_2^t\Big) \sqcup  \Big( \tau_t^{-1}(\rL(I^t,i)\cap I_2^t)\Big), $$
and that $\Psi_{t+1}$ induces a bijection from $\rL(I^{t+1},i)$ onto $\rL(I\dnw k, \Psi_{t+1}(i))=\rL(I\dnw k, \Psi_{t}(i))$.

Hence, by Lemma \ref{lem:pullbacking-system}, 
we can construct the $I^t$-system $\mathcal J^t$  by taking the pullback of the $I\dnw k$-system $\mathcal J _k$ through $\Psi_{t}$, for all $t=0,\dots ,n-1$. 
It is easily seen that the pair $(I_1^t,I_2^t)$ is a compatible 
$\mathcal J ^t$-system  with respect to the isomorphism $\tau_t\colon I_1^t\to I_2^t$ (see Definition \ref{def:comp-ideals}). 
  
Notice that $\mathcal J ^{t+1}$ coincides with the $I^{t+1}$-system $(\mathcal J^t)'$ described in Definition \ref{def:crownsyst} 
(using the $\mathcal J^t $-compatible pair $(I_1^t,I_2^t)$). Therefore, 
by Lemma \ref{lem:pushouting-I-systems}, $M(\mathcal J^{t+1})$ is a crowned pushout of $M(\mathcal J^t)$. Observe that $M(\mathcal J^n)=M(\mathcal J _k)$. 
  
So, $M(\mathcal J _k)$ can be obtained from $M(\mathcal J_{\mathbb F (k)})$ by a sequence of crowned pushouts. Since the poset $\mathbb F (k)$ satisfies the condition that $[p,k]$ is a chain for every $p\in \mathbb F (k)$
({\bf \ref{pt:KeyPoint}}), Proposition \ref{prop:chins-up} proves that $M(\mathcal J _{\mathbb F (k)})$  is a refinement monoid. 
Therefore, by Proposition \ref{crownpushmon}, $M(\mathcal J ^t)$ is a conical refinement monoid for all $t=0,1,\dots ,n$. In particular $M(\mathcal J_k)$ is a refinement monoid.
 
In order to extend this result to $I$, we will apply a similar strategy to the onto poset map $\Psi:\bigsqcup_{k\in \Ma (I)} (I\dnw k) \to I$. 
For this, we produce, by recurrence on $t$, a family $(I^t)_{0\leq t\leq s}$ of posets and a family $\psi_t: I^t\rightarrow I^{t+1}$ of onto poset maps, 
starting with  $I^0=\bigsqcup_{k\in \Ma (I)} (I\dnw k)$ and ending with $I^s=I$, satisfying the properties stated in Lemma \ref{lem:6.5for Posets}. Let us 
illustrate the procedure with the first step. For this, we enumerate $\Ma (I)=\{k_0, k_1, \dots ,k_s\}$. We define 
$I_1^0=(I\downarrow k_0)\cap (I\downarrow k_1)\subset (I\downarrow k_0)$ and $I_2^0=(I\downarrow k_0)\cap (I\downarrow k_1)\subset (I\downarrow k_1)$ (that is, we look at this intersection, first in the
disjoint copy of $I\downarrow k_0$ in $\bigsqcup_{k\in \Ma (I)} (I\dnw k)$, and then in the disjoint copy of $I\downarrow k_1$ in $\bigsqcup_{k\in \Ma (I)} (I\dnw k)$). 
Clearly, $I_1^0$ and $I_2^0$ are disjoint lower sets of $\bigsqcup_{i\in \Ma (I)} (I\dnw i)$, the map $\Psi$ restricted to $I_j^0$ ($j=1,2$) is injective, 
and $\Psi(I_1^0)=\Psi(I_2^0)$. So, $\tau_0  =(\Psi|_{I_2^0}^{-1}\circ \Psi|_{I_1^0}): I_1^0\rightarrow I_2^0$ is an isomorphism.
Now we define $I^1=(I^0)'= I^0\setminus I^0_2$, with the order $\le '$ determined as in Lemma \ref{lem:6.5for Posets}(c2). Let $\psi _0\colon I^0 \to I^1$ be the canonical identification map, 
and  let $\Psi_1\colon I^1\to I$ be the restriction map.  Proceeding in this way, we obtain the desired family of posets and the desired family of onto poset maps.

Note that the pullback of $\mathcal J$ with respect to $\Psi $ is precisely $\bigsqcup _{k\in \Ma (I)} \mathcal J_k$, and that the corresponding monoid is 
$$M(\bigsqcup _{k\in \Ma (I)} \mathcal J_k) =\prod_{k\in \Ma (I)} M(\mathcal J_k).$$
Moreover, the map induced by $\Psi $ is just the natural map $\prod_{k\in \Ma (I)} M(\mathcal J_k) \to M(\mathcal J)$ induced by the inclusions of the order-ideals
$M(\mathcal J_k)$ into $M(\mathcal J)$. Now, the same proof that we have used above shows that  
$M(\mathcal J)$ can be obtained from $\prod_{k\in \Ma (I)} M(\mathcal J _k)$ by 
a finite sequence of crowned pushouts. By the above argument, 
$M(\mathcal J _k) $ is a refinement monoid for every $k\in \Ma (I)$, and thus so is $\prod _{k\in \Ma (I)} M(\mathcal J _k) $. 
Therefore, we can conclude from Proposition \ref{crownpushmon} that $M(\mathcal J)$ has refinement. This concludes the proof.
\end{proof}

\section{Tameness and refinement property of $M(\mathcal J)$ for arbitrary $I$-systems}
\label{sect:ref-arbitrary}

In this section we will prove that the monoids associated to $I$-systems are tame refinement monoids (Theorem \ref{thm:tameness}). 
We will proceed to develop the proof through several intermediate steps. 
We seek to apply \cite[Theorem 2.6]{AraGood}, so our aim is to build, given an $I$-system $\mathcal J$ and a 
finitely generated submonoid $M'$ of $M(\mathcal J)$, a finitely generated {\it refinement} monoid $\mathcal M$  and monoid homomorphisms
$\gamma: M'\rightarrow \mathcal M $ and $\delta : \mathcal M\rightarrow M(\mathcal{J})$ such that $\delta \circ \gamma =\mbox{Id}_{M'}$. 
Note that, in order to achieve this, we can replace $M'$ be any larger submonoid of $M(\mathcal J)$. The larger submonoid needed for the proof  
will be of the form considered in Lemma \ref{lem:Pas1}. Lemmas \ref{lem:enlarging-K} and \ref{lem:order-induction} provide the technical ingredients
needed to build a suitable finitely generated refinement monoid $\mathcal M$ and suitable maps $\gamma, \delta$. 

\begin{lem}\label{lem:Pas1}
Let $I$ be an arbitrary poset, let $\mathcal{J}$ be an $I$-system, let $M:=M(\mathcal{J})$ be the associated conical monoid, and let $M'$ be a finitely generated submonoid of $M$. 
Then, there exists a finite subset $K$ of $I$ and a family of subsemigroups $S_i$ of $M_i$ for all $i\in K$ such that: 
\begin{enumerate}
\item $M'\subseteq \{0\}\sqcup \Big( \sum\limits _{i\in K} \chi_i (S_i)\Big) $.
\item There exist  finitely generated subgroups $X_i$ of $G_i$, $i\in K$, such that
$$S_i=\left\{
\begin{array}{cc}
  X_i ,  & \text{if  } i\in K_{reg}  \\
  \N\times X_i ,  &   \text{if  } i\in K_{free}
\end{array}\right.
$$
\end{enumerate}
\end{lem}

\begin{proof} This is clear from Corollary \ref{cor:presentationofM} and the form of the semigroups $M_i$ for $i\in I$.
\end{proof}

We want to enlarge the data in Lemma \ref{lem:Pas1} to a larger set of data, so that the relations satisfied by the monoid $\{ 0 \}\sqcup \Big( \sum\limits_{i\in K} \chi _i(S_i)\Big)$ are ``explained''
by the new monoid we are going to build. For this, it is convenient to use the associated partial order of groups $((\widehat{G}_i)_{i\in I},\widehat{\varphi}_{ji} (i<j))$
(see Section \ref{section:Isystems}). For any poset $I$ and any $I$-system $\mathcal J$, $\widetilde{M}(\mathcal J)$ stands for the conical regular monoid associated to
the partial order of groups $((\widehat{G}_i)_{i\in I},\widehat{\varphi}_{ji} (i<j))$. Recall that $M(\mathcal J)\subseteq \widetilde{M}(\mathcal J)$ (see Definition \ref{def:MJ} and Corollary \ref{cor:NouSubmonoid}).

Recall the semilattice of abelian groups $\Big( (\widehat{H}_a)_{a\in A(I)}, f_{a}^b \, \, (a\subset b) \Big)$ associated in Section 1 to the partial order of groups $((\widehat{G}_i)_{i\in I}, \widehat{\varphi} _{ji})$.
There is a corresponding associated monoid $MH (\mathcal J)= \bigsqcup\limits _{a\in A(I)} \widehat{H}_a$, and a canonical surjective monoid homomorphism
$$\Phi\colon MH(\mathcal J)\to \widetilde{M}(\mathcal J).$$

Let $K$ be a finite subset of $I$, and let $(X_i)_{i\in K}$ be a family of finitely generated groups, with $X_i$ a subgroup of $G_i$ for each $i\in K$.
 Let $\widehat{X}_i$ be the corresponding (finitely generated) subgroup of $\widehat{G}_i$ (so that $\widehat{X}_i= X_i$ if $i\in I_{{\rm reg}}$ and 
 $\widehat{X}_i=\Z\times  X_i$ if $i\in I_{{\rm free}}$). 
Consider the monoid
$$F= \{0\} \sqcup \Big( \bigoplus _{i\in K} \widehat{X}_i \Big) .$$
We have an obvious homomorphism $f\colon F\to MH(\mathcal J)$ sending $x\in \widehat{X}_i$ to $\chi(I\dnw i, i,x)\in \widehat{H}_{I\dnw i}$.

\begin{lem}
 \label{lem:enlarging-K}
 In the situation described before, 
there exists a finite subset $I'$ of $I$ containing $K$, and a family of finitely generated 
 subgroups $G_i'$ of $G_i$, for $i\in I'$, with $X_i\subseteq G_i'$ for $i\in K$, such that $\ker (\Phi \circ f)$
 is generated by a finite set of elements $\mathcal F$ satisfying the following property: For each $(x,y)\in \mathcal F$ there is a unique $a\in A(I)$ 
 such that $f(x), f(y)\in \widehat{H}_{a}$, and 
$f(x)-f(y)$ belongs to the subgroup of $\widehat{H}_a$ generated by the elements $\chi(a, i, g)-\chi (a,j,\widehat{\varphi}_{ji}(g))$, with
$g\in \widehat{G}_i'$ and $i,j\in I'$, $i<j\in a$. 
 \end{lem}

\begin{proof}
Since $F$ is a finitely generated abelian monoid, it follows from Redei's Theorem \cite{Freyd} that $\ker (\Phi \circ f)$ is a finitely generated congruence.
So, there is a finite set $\mathcal F$ of elements generating $\ker (\Phi \circ f)$. 
 For $(x,y)\in \mathcal F$,  $\Phi(f(x))= \Phi(f(y))\in \widehat{H}_a/U_a$ for a unique $a\in A(I)$. Therefore $f(x),f(y)\in \widehat{H}_a$,
 and $f(x)-f(y)$ is a finite sum of elements of the form $\pm (\chi (a,i,u)-\chi (a,j,\widehat{\varphi}_{ji}(u))) $, for $u\in \widehat{G}_i$ and $i<j\in a$. 
 Now, let $I'$ be the union of $K$ and the (finite) support of all these elements. For $i\in K$, let $G_i'$ be the subgroup of $G_i$ generated by $X_i$ and 
 the $G_i$-components of the elements of $\widehat{G}_i$ appearing in the above expressions (that is, elements $u$ in $\widehat{G}_i$ such that
 $\chi (a,i,u)-\chi (a,j,\widehat{\varphi}_{ji}(u))$ appears in the expression of $f(x)-f(y)$ for some $(x,y)\in \mathcal F$, 
 and elements of the form $\widehat{\varphi}_{ij}(u)$, where $u\in \widehat{G}_j$, $j<i$, and
 $\chi (a,j,u)-\chi (a,i,\widehat{\varphi}_{ij}(u))$ appears in the expression of $f(x)-f(y)$ for some $(x,y)\in \mathcal F$). 
Similarly, for $i\in I'\setminus K$, let $G_i'$ be the subgroup of $G_i$ generated by the $G_i$-components of the elements of $\widehat{G}_i$ appearing in the above expressions.
   \end{proof}
   
 \medskip

The subset $I'$ of $I$ obtained in Lemma \ref{lem:enlarging-K} will be considered as a poset with the order $\le$ inherited from $(I,\le )$. Now, for any pair $i,j\in I$ with $j<i$, we define an auxiliary subgroup $S_{ij}$ of $G_i$. 

\begin{defi}\label{def:S_ij}
{\rm Let $I$ be a poset,  and let $\mathcal{J}=(I, (G_i)_{i\in I}, \varphi_{ij} (j<i))$ be an $I$-system. Then, for any $i,j\in I$ with $j<i$, we define a group $S_{ij}$ as follows:
\begin{enumerate}
\item If $j$ is regular, we define $S_{ij}$ to be the trivial subgroup of $G_i$.
\item If $i$ is regular and $j$ is free, we define $S_{ij}$ to be the subgroup of $G_i$ generated by $\varphi_{ij}(1,e_j)$.
\item If both $i$ and $j$ are free then, by condition (c2) in Definition \ref{def:I-system}, there are a finite subset $T_{ij}\subset I$ and elements $\{ z^{(ij)}_t : t\in T_{ij} \}$ with $t<i$, and $z^{(ij)}_t\in M_t$ for all $t\in T_{ij}$, such that 
$$-\varphi_{ij} (1,e_j) = \sum\limits_{t\in T_{ij}} \varphi _{it} (z^{(ij)}_t) .$$
We define $S_{ij}$ to be the subsemigroup of $G_i$ generated by 
$$\{\varphi_{ij} (1,e_j) \}\cup \Big( \{ \varphi_{it} (z_t^{(ij)} ): t\in T_{ij} \} \Big).$$ 
Note that $S_{ij}$ is indeed a {\it finitely generated subgroup} of $G_i$, since it contains the inverse of each one of its generators.
\end{enumerate}}
\end{defi}

\begin{rema}\label{rem:S_ij dins G_i'}
{\rm Observe that we can assume, without loss of generality, that the finitely generated subgroups $G_i'$, for $i\in I'$, obtained in Lemma \ref{lem:enlarging-K},  
satisfy that $\sum\limits _{j<i, j\in I'} S_{ij}\subseteq G_i'$. }
\end{rema}

\medskip
 
We now state a crucial lemma.

\begin{lem}
 \label{lem:order-induction} Let $I'$ and $\{ G_i' :i\in I' \}$ be as stated in Lemma \ref{lem:enlarging-K}, and assume that $\sum\limits _{j<i, j\in I'} S_{ij}\subseteq G_i'$
 for all $i\in I'$. Then, there exists a finite poset $$(I'',\le ') = (I'\sqcup J', \le')$$ 
 such that
 \begin{enumerate}
  \item The elements of $J'$ are pairwise incomparable minimal elements of $I''$,
  \item The order $\le'$ agrees with the original order $\le$ on $I'$,
   \end{enumerate}
and there exist a family of finitely generated subgroups $(G_i'')_{i\in I'}$, with $G_i'\subseteq G_i''\subseteq G_i$ for all $i\in I'$, a map
$\tau\colon J'\to (I\setminus I')$ such that, for $j\in J'$ and $i\in I'$, we have  $j\le' i \implies \tau (j)< i$ in $I$,
and elements $x_j\in M_{\tau (j)}$, $j\in J'$, such that 
\begin{enumerate}
 \item[(a)] $G_i'' = \sum\limits _{j<i,\, j\in I'}\varphi_{ij}(M_j'') + \sum\limits _{j\le' i,\, j\in J'} \langle \varphi_{i\tau (j)} (x_j) \rangle \, \,$
  for all $i\in I'_{{\rm free}}$
 \item[(b)] $G_i'' \supseteq  \sum\limits _{j<i,\, j\in I'} \varphi_{ij}(M_j'') + \sum\limits _{j\le' i,\,  j\in J'} \langle \varphi_{i\tau (j)} (x_j) \rangle \,\,$
  for all $i\in I'_{{\rm reg}}$
 \end{enumerate}
where, for $i\in I'$, we set 
$$M_i''=\left\{
\begin{array}{cc}
  G_i'' ,  & \text{if  } i\in (I')_{reg}  \\
  \N\times G_i'' ,  &   \text{if  } i\in (I')_{free}
\end{array}\right.
$$
 \end{lem}
\begin{proof} We will show by (order-)induction the following statement:

\medskip

{\it Let $\mathcal U$ be an upper subset of $I'$. Then there exists a finite poset $$(I_{\mathcal U},\le_{\mathcal U}) = (I'\sqcup J_{\mathcal U}, \le_{\mathcal U})$$ 
 such that
 \begin{enumerate}
  \item The elements of $J_{\mathcal U}$ are pairwise incomparable minimal elements of $I_{\mathcal U}$,
  \item The order $\le_{\mathcal U}$ agrees with the original order $\le$ on $I'$,
  \item For $j\in J_{\mathcal U}$ and $i\in I'$ we have $j\le_{\mathcal U} i\implies i\in \mathcal U$, 
 \end{enumerate}
and there exist a family of finitely generated subgroups $(G_i^{\mathcal U})_{i\in I'}$, with $G_i'\subseteq G_i^{\mathcal U}\subseteq G_i$ for all $i\in I'$, a map
$\tau_{\mathcal U}\colon J_{\mathcal U}\to (I\setminus I')$ such that, for $j\in J_{\mathcal U}$ and $i\in I'$, we have  $j\le_{\mathcal U} i \implies \tau_{\mathcal U} (j)\le i$ in $I$,
and elements $x_j\in M_{\tau_{\mathcal U} (j)}$, $j\in J_{\mathcal U}$, such that 
\begin{enumerate}
 \item[(a)] $G_i^{\mathcal U} = \sum\limits _{j<i,\, j\in I'}\varphi_{ij}(M_j^{\mathcal U}) + \sum\limits _{j\le_{\mathcal U} i,\, j\in J_{\mathcal U}} \langle \varphi_{i\tau_{\mathcal U}(j)} (x_j) \rangle $
 for all $i\in \mathcal U_{{\rm free}}$
 \item[(b)] $G_i^{\mathcal U} \supseteq  \sum\limits _{j<i,\, j\in I'} \varphi_{ij}(M_j^{\mathcal U}) + \sum\limits _{j\le_{\mathcal U} i,\,  j\in J_{\mathcal U}} \langle \varphi_{i\tau_{\mathcal U}(j)} (x_j) \rangle $
 for all $i\in \mathcal U_{{\rm reg}}$
 \end{enumerate}
where, for $i\in I'$, we set 
$$M_i^{\mathcal U}=\left\{
\begin{array}{cc}
  G_i^{\mathcal U} ,  & \text{if  } i\in (I')_{reg}  \\
  \N\times G_i^{\mathcal U} ,  &   \text{if  } i\in (I')_{free}
\end{array}\right.
$$}

\medskip

Once this is done, the statement in the lemma follows by taking $J':=J_{I'}$, $\, \le ' := \le_{I'}$, $\, \tau := \tau_{I'}$, and $G_i'':= G_i^{I'}$ for all $i\in I'$.  

\medskip

We start with $\mathcal U =\emptyset$. In this case we set $J_{\emptyset} = \emptyset$, so that $I_{\emptyset}= I'$ with the order
$\le $ induced from $I$, and we set $G_i^{\emptyset} = G_i'$. 

Assume that $\mathcal U$ is an upper subset of $I'$ for which we have defined $I_{\mathcal U}= I'\sqcup J_{\mathcal U}$, together with the partial order $\le _{\mathcal U}$
which satisfies the stated conditions (1)--(3), the map $\tau_{\mathcal U}$, subgroups $G_i^{\mathcal U}$, $i\in I'$ and elements $x_j\in M_{\tau_{\mathcal U}(j)}$, $j\in J_{\mathcal U}$ 
satisfying conditions (a),(b).
Let $i_0$ be a maximal element in $I'\setminus \mathcal U$. We will build the corresponding objects for the upper subset $\mathcal U ':= \mathcal U \cup \{i_0\}$.

Assume first that $i_0$ is regular. Then, we set $J_{\mathcal U '}= J_{\mathcal U}$, $\tau _{\mathcal U'}= \tau_{\mathcal U}$,  $\le _{\mathcal U '}=\le _{\mathcal U}$, and $G_i^{\mathcal U'}= G_i^{\mathcal U}$
for $i\in I'\setminus \mathcal U'$. 

We define 
$$G_{i_0}^{\mathcal U'}= G_{i_0}^{\mathcal U}+\sum\limits _{i<i_0, i\in I'} \widehat{\varphi} _{i_0i} (G_i^{\mathcal U'}) .$$
For $i\in \mathcal U$,  we define inductively $G_i^{\mathcal U'}$ by
$$G_{i}^{\mathcal U'}= G_{i}^{\mathcal U}+\sum\limits _{j<i, j\in I'} \widehat{\varphi} _{ij} (G_j^{\mathcal U'}) .$$
We have to check condition (b) for $i_0$ and conditions (a) or (b) for $i\in \mathcal U$ according to whether $i$ is free or regular respectively.

Note that condition (b) for $i_0$ reads 
$$G_{i_0}^{\mathcal U'} \supseteq \sum\limits _{j<i_0, j\in I'} \varphi_{i_0j}(M_j^{\mathcal U}).$$
(Use condition (3) and the facts that $J_{\mathcal U'}=J_{\mathcal U}$ and $\le_{\mathcal U'}=\le_{\mathcal U}$).
For $j\in I'$ with $j<i_0$, since $S_{i_0, j}\subseteq G'_{i_0}\subseteq G^{\mathcal U'}_{i_0}$, we only need to show that $\widehat{\varphi} _{i_0j} (G_{j}^{\mathcal U}) \subseteq G_{i_0}^{\mathcal U'}$,
but this is obvious from the definition.

If $i\in \mathcal U_{{\rm free}}$, then (a) follows from the induction hypothesis and the observation that, for $j<i$, $j\in I'$, we have
$S_{ij}+\varphi _{ij}(M_j^{\mathcal U'}) =S_{ij}+ \widehat{\varphi}_{ij} (G_j^{\mathcal U'})$. The proof of (b) in case $i\in \mathcal U_{{\rm reg}}$ is similar.

We now consider the case where $i_0$ is free. Since $\mathcal J$ is an $I$-system and $G_{i_0}^{\mathcal U}$ is finitely generated, 
there is a finite subset $I^{(i_0)}$ of $\{ j\in I : j<i_0 \}$
and finitely generated subsemigroups $N_j'$ of $M_j$, for $j\in I^{(i_0)}$, such that
\begin{equation}
 \label{eq:i0-inculison}
G_{i_0}^{\mathcal U} \subseteq \sum\limits _{j\in I^{(i_0)}} \varphi_{i_0j} (N_j') .
 \end{equation}

Recall from the construction of $S_{ij}$ for $j<i$ with $i,j\in I_{{\rm free}}$ (Definition \ref{def:S_ij}(3)), that $S_{ij}$ is 
the semigroup generated by the elements $\varphi_{ij} (1,e_j)$ and $\varphi _{it}(z_t^{(ij)})$, where $z_t^{(ij)}\in M_t$, $t\in T_{ij}$,  and that this semigroup is indeed a group. 
We denote by $\widetilde{z}^{(ij)}_t$ the group component of $z_t^{(ij)}\in M_t$, that is $z_t^{(ij)}= (n_t^{(ij)}, \widetilde{z}_t^{(ij)})\in \N\times G_t$ if $t$ is free and
$z_t^{(ij)}= \widetilde{z}_t^{(ij)}\in M_t$ if $t$ is regular. We use a similar notation for the generators of each $N_j'$, so letting $x_{j,1}, \dots , x_{j,n_j}$ be a finite set of semigroup generators 
of $N_j'$, we denote by $\widetilde{x}_{j,t}\in G_j$ the group component of each $x_{j,t}$, so that $x_{j,t} = (n_{j,t}, \widetilde{x}_{j,t})$ if $j$ is free and $x_{j,t} = \widetilde{x}_{j,t}$ if $j$ 
is regular. We will denote by $\widetilde{N}_j'$ the subset of $G_j$ of group elements $\{\widetilde{x}_{j,1}, \dots , \widetilde{x}_{j,n_j}\}$ of the semigroup generators $\{x_{j,1}, \dots , x_{j,n_j}\}$ of $N_j'$.

Let $J_{i_0}:= \bigcup_{j\in I^{(i_0)}_{{\rm free}}} T_{i_0,j}$ be the support of the elements $z_t^{(i_0j)}$ with $j\in I^{(i_0)}_{{\rm free}}$. 
Note that $J_{i_0}$ is a finite subset of $I$, and that $t<i_0$ for all $t\in J_{i_0}$. We will denote $\widetilde{R}_{i_0}=\{ \widetilde{z}_j^{(i_0,j')}: j'\in I^{(i_0)}_{{\rm free}}, \, j\in T_{i_0,j'}\}$.

We first define $G_j^{\mathcal U'}$ for $j\in I'\setminus \mathcal U'$. We use the notation $\text{Gp}(X)$ to denote the subgroup generated by a subset $X$ of
a group $G$.

$\bullet $ If $j\in I'\setminus \mathcal U '$ and $j\notin J_{i_0}\cup I^{(i_0)}$, then set $G_j^{\mathcal U'}:=G_j^{\mathcal U}$.

$\bullet $ If $j\in I'\setminus \mathcal U '$, $j\notin J_{i_0}$,  and $j\in  I^{(i_0)}$, then set $G_j^{\mathcal U'}:=G_j^{\mathcal U}+ \text{Gp}(\widetilde{N}_j')$.

$\bullet $ If $j\in I'\setminus \mathcal U '$, $j \in J_{i_0}$ and $j\notin I^{(i_0)}$, then set $G_j^{\mathcal U'}:=G_j^{\mathcal U}+\text{Gp}(\widetilde{R}_{i_0})$.

$\bullet $ If $j\in I'\setminus \mathcal U '$ and $j\in J_{i_0}\cap I^{(i_0)}$, then set $G_j^{\mathcal U'}:=G_j^{\mathcal U}+\text{Gp}(\widetilde{N}_j')+ 
\text{Gp}(\widetilde{R}_{i_0})$.
\vspace{.2truecm}

It is convenient at this point to introduce the following set:
$$Z_{i_0}:= \{ (j,s)\in I^{(i_0)}_{{\rm free}}\times (J_{i_0}\setminus I') : s\in T_{i_0,j} \}.$$
In words, $Z_{i_0}$ is the set of all ordered pairs $(j,s)$ such that $j\in I^{(i_0)}_{{\rm free}}$ and $s\in T_{i_0,j}\setminus I'$.

Now, we define 
%
$$J_{\mathcal U'}:= J_{\mathcal U}\, \sqcup \Big(\bigsqcup_{j\in I^{(i_0)}\setminus I'} \{ v_{j,\pm t}: 1\le t\le n_j \} \Big) \, \sqcup \{ u_{j} : j\in I^{(i_0)}_{{\rm free}}\setminus I' \}\, 
\sqcup \{ w_{(j,s)} : (j,s) \in Z_{i_0}\} .$$
The new order $\le _{\mathcal U'}$ on $I_{\mathcal U'}:= I'\sqcup J_{\mathcal U'}$ is defined by extending the order $\le_{\mathcal U}$ on $I_{\mathcal U}$ and adding the new relations:
\begin{itemize}
\item[(i)] $v_{j,\pm t} <_{\mathcal U'} i$, when $j\in I^{(i_0)}\setminus I'$, $i\in I'$ and $i_0\le i$,.\item[(ii)]$u_j<_{\mathcal U'} i$, when $j\in I^{(i_0)}_{{\rm free}}\setminus I'$, $i\in I'$ and $i_0\le i$.
\item[(iii)]$ w_{(j,s)} <_{\mathcal U'} i$, when $(j,s)\in Z_{i_0}$ and $i_0\le i$. 
\end{itemize}
Note that the new order enjoys $(1$-$3)$.

Define $\tau _{\mathcal U'}\colon  J_{\mathcal U'}\to I\setminus I'$ by:
\begin{itemize}
\item[(i)] $ \tau_{\mathcal U'}= \tau_{\mathcal U}$ on $J_{\mathcal U}$.
\item[(ii)] $\tau_{\mathcal U'}(v_{j,\pm t})=  j$ for $j\in I^{(i_0)}\setminus I'$. 
\item[(ii)]  $\tau_{\mathcal U'} (u_j)=j$ for $j\in I^{(i_0)}_{{\rm free}}\setminus I'$.  
\item[(iv)] $\tau_{\mathcal U'} (w_{(j,s)}) = s$ for $(j,s)\in Z_{i_0}$.
\end{itemize} 
Observe that, for $\alpha\in J_{\mathcal U'}$ and $i\in \mathcal U'$, we have $\alpha \le_{ \mathcal U'} i \implies \tau_{\mathcal U'}(\alpha )< i$. Indeed, this is clear by induction on $J_{\mathcal U}$, and for
$v_{j,t}$, $u_j$, and $w_{(j,s)}$ it follows from the facts that $j< i_0\le i$ whenever $j\in I^{(i_0)}$ and that $s<i_0\le i$ whenever $s\in J_{i_0}\setminus I'$.

The elements $x_j$ for $j\in J_{\mathcal U}$ are defined to be the same elements $x_j^{\mathcal U}$ which were previously defined by induction, and 
$$x_{v_{j,\pm t}} :=  (n_{j,t}, \pm \widetilde{x}_{j,t}),\qquad  x_{u_j}= (1,e_j), \qquad x_{w_{(j,s)}} := z_s^{(i_0,j)} .$$
Note that $x_{\alpha}\in M_{\tau_{\mathcal U'}(\alpha )}$ for all $\alpha\in J_{\mathcal U'}$.

Finally, we set
$$G_{i_0}^{\mathcal U'}= G_{i_0}^{\mathcal U} + \sum\limits _{j<i_0, j\in I'} \widehat{\varphi}_{i_0j}(G_j^{\mathcal U'}) + \sum\limits _{j\in I^{(i_0)}\setminus I'}\sum\limits _{t=1}^{n_{j}}\langle  \widehat{\varphi}_{i_0j} (\pm \widetilde{x}_{j,t})\rangle +$$
$$+\sum\limits _{j\in I^{(i_0)}_{{\rm free}}\setminus I'} \langle \varphi_{i_0j} (1,e_j) \rangle + \sum\limits _{(j,t)\in Z_{i_0}} \langle \varphi_{i_0t} (z_t^{(i_0,j)}) \rangle .$$
Note that $S_{i_0,j}\subseteq G_{i_0}^{\mathcal U'}$ for all $j\in I^{(i_0)}$. So, $G_{i_0}^{\mathcal U'}$ is a group.

We have to verify condition (a) in Lemma \ref{lem:order-induction} for $G_{i_0}^{\mathcal U'}$. Let us denote by $A$ the right hand side of that formula.
Note that 
$$A= \sum\limits _{j<i_0,\, j\in I'}\varphi_{i_0 j}(M_j^{\mathcal U'}) + \sum\limits _{\alpha \in J_{\mathcal U'}\setminus J_{\mathcal U}}\langle \varphi_{i_0\tau_{\mathcal U'}(\alpha)} (x_{\alpha})\rangle,$$
because if $\alpha\in J_{\mathcal{U}}$ then $\tau_{\mathcal{U}'}(\alpha)\not< i_0$.
We first check that $A\subseteq G_{i_0}^{\mathcal U'}$. For $j\in I'$ with $j<i_0$, since $S_{i_0,j}$ and $\widehat{\varphi}_{i_0j}(G_j^{\mathcal U '})$ are contained in 
$G_{i_0}^{\mathcal U'}$, we get that $\varphi _{i_0j} (M_j^{\mathcal U'})\subseteq G_{i_0}^{\mathcal U'}$.
Similarly, for $j\in I^{(i_0)}\setminus I'$, we have that $S_{i_0,j}\subseteq G_{i_0}^{\mathcal U'}$, and so 
$$\varphi_{i_0j}(x_{v_{j,\pm t}})=\varphi_{i_0j}((n_{j,t}, \pm \widetilde{x}_{j,t}))\in G_{i_0}^{\mathcal U'}.$$
It is obvious that $\varphi_{i_0j} (x_{u_j})= \varphi_{i_0j}(1,e_j)$ for $j\in I^{(i_0)}_{{\rm free}}\setminus I'$, and $\varphi _{i_0t}(x_{w_{(j,t)}})=\varphi_{i_0t} (z_t^{(i_0,j)})$ 
for $j\in I^{(i_0)}_{{\rm free}}$ and $t\in T_{i_0,j}\setminus I'$, 
belong to $G_{i_0}^{\mathcal U'}$.

Conversely, we now show that $G_{i_0}^{\mathcal U'}\subseteq A$. Again, the choice of the groups $G_j^{\mathcal U'}$ and the elements $x_{\alpha}$ makes it 
clear that $S_{i_0,j}\subseteq A$ for all $j\in I^{(i_0)}$. Indeed, given $j \in I^{(i_0)}_{{\rm free}}$, we have two possibilities: 
\begin{itemize}
\item[(i)] If $j\in I'$, then $(1_j, e_j) \in M_j^{\mathcal U'}$, 
and thus $\varphi _{i_0j}(1,e_j)\in A$. 
\item[(ii)] If $j\notin I'$, then $x_{u_j}= (1,e_j)$, and thus $\varphi _{i_0j}(1,e_j)\in A$. 
\end{itemize}
Now, if 
$t \in T_{i_0,j}\cap I'$, then $z_t^{(i_0,j)}\in M_t^{\mathcal U'}$, as $\tilde{z}_t^{(i_0j)}\in G_t^{\mathcal U'}$, while if $t\in  T_{i_0,j}\setminus I'$, then $(j,t)\in Z_{i_0}$ and so 
$z_t^{(i_0,j)}= x_{w_{(j,t)}}$ and $\varphi _{i_0t}(z_t^{(i_0,j)})\in A$.

By (\ref{eq:i0-inculison}) and the choice of the groups $G_j^{\mathcal U'}$ for $j\in I^{(i_0)}\cap I'$, we have
$$G_{i_0}^{\mathcal U} \subseteq \sum\limits_{j\in I^{(i_0)}} \varphi_{i_0j}(N_j')\subseteq \sum\limits _{j\in I^{(i_0)}\cap I'} \Big( S_{i_0,j}+\widehat{\varphi}_{i_0j}(\widetilde{N}_j') \Big) +
\sum\limits_{j\in I^{(i_0)}\setminus I'} \varphi_{i_0j}(N_j')\subseteq A.$$
In particular, we obtain that $S_{i_0,j}\subseteq A$ for all $j\in I^{(i_0)}$ and for all $j\in I'$ such that $j<i_0$ (use that $S_{i_0,j}\subseteq G_{i_0}'\subseteq G_{i_0}^{\mathcal U}$ for $j\in I'$ with $j<i_0$). From this we easily obtain that 
$\widehat{\varphi}_{i_0j}(G_j^{\mathcal U'}) \subseteq A$ for all $j\in I'$ such that $j<i_0$, and that $\widehat{\varphi}_{i_0j} (\pm \widetilde{x}_{j,t})\in A$ for all
$j\in I^{(i_0)}\setminus I'$ and all $t=1,\dots ,n_j$. The elements $\varphi_{i_0j}(1,e_j)$, for $j\in I^{(i_0)}_{{\rm free}}\setminus I'$, and $\varphi_{i_0t}(z_t^{(i_0,j)})$, for 
$(j,t)\in Z_{i_0}$, 
trivially belong to $A$.
This concludes the proof of (a) for $i_0$ and $\mathcal U'$.

Now  we define inductively $G_i^{\mathcal U'}$ for $i\in \mathcal U$. Assuming we have already defined the groups $G_{i'}^{\mathcal U'}$ for $i'\in \mathcal U$ with $i'<i$, 
(and that those satisfy (a) or (b) depending on whether $i'$ is free or regular), define $G_i^{\mathcal U'}$ by the formula
$$G_{i}^{\mathcal U'} := G_i^{\mathcal U} + \sum\limits _{j<i,j\in I'} \widehat{\varphi}_{ij} (G_j^{\mathcal U'}) .$$
We have to show condition (a) or (b) for $i\in \mathcal U$ and $G_i^{\mathcal U'}$, depending on whether $i$ is free or regular. Assume that $i$ is free.
Then (a) reads
$$G_i^{\mathcal U'}=  \sum\limits _{j<i,\, j\in I'}\varphi_{ij}(M_j^{\mathcal U'}) + \sum\limits _{j\le_{\mathcal U'} i,\, j\in J_{\mathcal U'}} \langle \varphi_{i\tau_{\mathcal U'}(j)} (x_j) \rangle .$$
Now, for $j\in J_{\mathcal U'}\setminus J_{\mathcal U}$, we have that $j\le_{\mathcal U'} i$ if and only if $i_0<i$, and, in this case, we have
$\tau_{\mathcal U'}(j)< i_0<i$. Therefore, we get
$$\varphi_{i\tau_{\mathcal U'}(j)}(x_j)= \widehat{\varphi}_{ii_0}(\varphi_{i_0 \tau_{\mathcal U'}(j)}(x_j))\in \widehat{\varphi}_{ii_0}(G_{i_0}^{\mathcal U'}) \subseteq G_i^{\mathcal U'}.$$
Using this and induction, it is easy to verify (a). 
If $i$ is regular, a similar argument  shows that (b) holds for $i$ and $\mathcal U'$. This concludes the proof.
\end{proof}

Now, we are ready to prove the main result in this section.

\begin{theor}
\label{thm:tameness}
Let $I$ be an arbitrary poset, and let $\mathcal J$ be an $I$-system. Then, the monoid $M(\mathcal J)$ is a tame refinement monoid.
\end{theor}
\begin{proof}
Let $I$ be an arbitrary poset, let $\mathcal{J}$ be an $I$-system, and let $M:=M(\mathcal{J})$ be the associated conical monoid. 

Now, we will show that 
for any finitely generated submonoid $M'$ of $M$ there exist a finite poset $I''$, a finitely generated $I''$-system $\mathcal{J}''$ and monoid homomorphisms
$\gamma: M'\rightarrow M(\mathcal{J}'')$ and $\delta : M(\mathcal{J}'')\rightarrow M(\mathcal{J})$ such that $\delta \circ \gamma =\mbox{Id}_{M'}$. 
By Proposition \ref{prop:finitely-generated}, Theorem \ref{theor:total-refinement} and \cite[Theorem 2.6]{AraGood}, this implies that $M$ is a tame refinement monoid.

Let $M'$ be a finitely generated submonoid of $M(\mathcal J)$. We can assume that $M' = \{ 0\}\sqcup \Big( \sum\limits _{i\in K} \chi_i (S_i)\Big)$, where
$S_i$ are as described in Lemma \ref{lem:Pas1}. Let $I'$ and $(G_i')_{i\in I'}$ be the larger finite subset of $I$, and  the family of groups, respectively,  built in
Lemma \ref{lem:enlarging-K}. We will also assume that  $\sum\limits _{j<i, j\in I'} S_{ij}\subseteq G_i'$
 for all $i\in I'$ (see Remark \ref{rem:S_ij dins G_i'}).

Consider the poset $(I'', \le ') = (I'\sqcup J',\le ')$, the family of groups $(G_i'')_{i\in I'}$, the map $\tau \colon J'\to (I\setminus I')$, and the elements
$x_j\in M_{\tau (j)}$, $j\in J'$, obtained in
Lemma \ref{lem:order-induction}. 
Now, we will define an $I''$-system $\mathcal J''$. First, we set $I''_{{\rm free}} = I'_{{\rm free}}\sqcup J'$
and $I''_{{\rm reg}}= I'_{{\rm reg}}$. The groups corresponding to $i\in I'$ are the groups $G_i''$. For $i\in J'$, set  $G''_i= \{ e_i\}$ (the trivial group). 
For $i,j\in I'$ with $j<i$, we define 
$$\varphi ''_{ij}\colon M_j''\longrightarrow G_i''$$
by $\varphi ''_{ij} = \varphi _{ij}|_{M_j''}$. Observe that this is well defined by (a) and (b) in Lemma \ref{lem:order-induction}. 
If $j\le ' i$ for  $j\in J'$ and $i\in I'$ then, again by Lemma \ref{lem:order-induction}, we must have $\tau (j) < i$ in $I$, so that 
we may define
$$\varphi ''_{ij} \colon \N= M_j'' \longrightarrow G_i''$$
by $\varphi ''_{ij}(1) = \varphi _{i \tau (j)}(x_j)$. By (a) and (b) in Lemma \ref{lem:order-induction}, we have that $\mathcal J''$ is an $I''$-system. \vspace{.3truecm}

Now, we shall use Lemma \ref{lem:enlarging-K} to show that there is a well defined monoid homomorphism
$$\phi \colon \{0\}\sqcup \Big( \sum\limits _{i\in K} \chi _i(\widehat{X}_i)\Big) \longrightarrow \widetilde{M}(\mathcal J''),$$
sending $\chi _i(g)\in \chi (\widehat{X}_i)$ to $\chi_i(g)\in \widehat{G}''_i$ for $g\in \widehat{X}_i$  
(where we look $\{0\}\sqcup \Big( \sum\limits _{i\in K} \chi _i(\widehat{X}_i)\Big)$ as a submonoid of $\widetilde{M}(\mathcal J)$). 
 Indeed, we have a monoid homomorphism 
 $$\phi_F\colon F =\{0\}\sqcup \Big(\bigoplus_{i\in K} \widehat{X}_i\Big) \longrightarrow \widetilde{M} (\mathcal J '') $$
 sending $g\in \widehat{X}_i$ to $\chi _i(g)\in 
 \widetilde{M}(\mathcal J'')$ for $i\in K$, and by the choice of the groups $G_i'$, $i\in I'$, we have that $\phi_F(x)=\phi_F(y)$ for all
 $(x,y)\in \mathcal F$, where $\mathcal F$ is the finite set of generators of $\ker (\Phi\circ f)$ coming from Lemma \ref{lem:enlarging-K}.
 Therefore, the map $\phi_F$ factorizes through $(\Phi\circ f) (F) = \{0\}\sqcup \Big( \sum\limits _{i\in K} \chi _i(\widehat{X}_i)\Big)$, and we obtain 
 a well-defined monoid homomorphism $\phi$ from $\{0\}\sqcup \Big( \sum\limits _{i\in K} \chi _i(\widehat{X}_i)\Big)$ to  $\widetilde{M}(\mathcal J'')$ as claimed. Observe that $\phi_F$ restricts to a monoid homomorphism from $M'\subseteq \{0\}\sqcup \Big( \sum\limits _{i\in K} \chi _i(\widehat{X}_i)\Big)$ 
to $M(\mathcal J'')\subseteq \widetilde{M} (\mathcal J'')$, sending $\chi_i(x)\in  \chi_i (S_i)$ to $\chi_i (x)\in M(\mathcal J'')$ for all $i\in K$ and all $x\in S_i$.  Let $\gamma \colon M'\to M(\mathcal J'')$ be this homomorphism.\vspace{.3truecm}

Finally,  we define a monoid homomorphism $\delta \colon M(\mathcal J'')\to  M$ by $\delta (\chi _i(m_i)) = \chi _i (m_i)$ for $m_i\in M_i''$
and $\delta (\chi _j(1))= \chi _{\tau (j)}(x_j)\in M_{\tau (j)}$ for $j\in J'$.
If $j\in J'$, $i\in I'$, and $j<' i$, then $\tau  (j)<i$ in $I$ so that, for $x\in M_i''$ we have 
$$\delta (\chi _i (x)) +\delta (\chi _j(1)) = \chi _i (x) + \chi _{\tau (j)}(x_j) = \chi _i(x+\varphi _{i\tau (j)} (x_j)) = \chi_i (x+\varphi_{ij}''(1))=
\delta (\chi_i (x+\varphi_{ij}''(1))).$$
By Corollary \ref{cor:presentationofM}, we thus get that $\delta$ is a well-defined monoid homomorphism. Clearly $\delta\circ \gamma = \iota_{M'}$.
This concludes the proof.
\end{proof}

\section*{Acknowledgments}

Part of this work was done during visits of the first author to the Departamento 
de Matem\'aticas de la Universidad de C\'adiz (Spain), and of the second author to the Departament de Matem\`atiques de la Universitat Aut\`onoma de Barcelona (Spain). 
Both authors thank both centers for their kind hospitality.

We also thank the anonymous referee for carefully reading the paper and for providing many helpful suggestions.

\end{document}